\documentclass[10pt,final,twocolumn]{IEEEtran}

\long\def\comment#1{}

\usepackage{textcomp}
\usepackage{cite}
\usepackage{graphicx}
\usepackage{subfigure}
\usepackage{psfrag}
\usepackage{url}
\usepackage{stfloats}
\usepackage{amsmath}
\usepackage{amssymb,amsfonts}
\usepackage{amsthm}
\usepackage{float}
\usepackage[usenames]{color}
\usepackage{algorithm,algorithmic}

\newtheorem{assumption}{Assumption}
\newtheorem{remark}{Remark}
\newtheorem{definition}{Definition}
\newtheorem{lemma}{Lemma}
\newtheorem{theorem}{Theorem}
\newtheorem{example}{Example}
\newtheorem{proposition}{Proposition}

\begin{document}

\setlength{\arraycolsep}{0.3em}

\title{Distributed Algorithms for Computing a Common Fixed Point of a Group of Nonexpansive Operators
\thanks{}}

\author{Xiuxian Li and Gang Feng
\thanks{X. Li and G. Feng are with Department of Biomedical Engineering, City University of Hong Kong, Kowloon, Hong Kong (e-mail: xiuxiali@cityu.edu.hk; megfeng@cityu.edu.hk).}
}

\maketitle

\setcounter{equation}{0}
\setcounter{figure}{0}
\setcounter{table}{0}

\begin{abstract}
This paper addresses the problem of seeking a common fixed point for a collection of nonexpansive operators over time-varying multi-agent networks in real Hilbert spaces, where each operator is only privately and approximately known to each individual agent, and all agents need to cooperate to solve this problem by propagating their own information to their neighbors through local communications over time-varying networks. To handle this problem, inspired by the centralized inexact Krasnosel'ski\u{\i}-Mann (IKM) iteration, we propose a distributed algorithm, called distributed inexact Krasnosel'ski\u{\i}-Mann (D-IKM) iteration. It is shown that the D-IKM iteration can converge weakly to a common fixed point of the family of nonexpansive operators. Moreover, under the assumption that all operators and their own fixed point sets are (boundedly) linearly regular, it is proved that the D-IKM iteration converges with a rate $O(1/k^{\ln(1/\xi)})$ for some constant $\xi\in(0,1)$, where $k$ is the iteration number. To reduce computational complexity and burden of storage and transmission, a scenario, where only a random part of coordinates for each agent is updated at each iteration, is further considered, and a corresponding algorithm, named distributed inexact block-coordinate Krasnosel'ski\u{\i}-Mann (D-IBKM) iteration, is developed. The algorithm is proved to be weakly convergent to a common fixed point of the group of considered operators, and, with the extra assumption of (bounded) linear regularity, it is convergent with a rate $O(1/k^{\ln(1/\xi)})$. Furthermore, it is shown that the convergence rate $O(1/k^{\ln(1/\xi)})$ can still be guaranteed under a more relaxed (bounded) power regularity condition.
\end{abstract}


\section{Introduction}\label{s1}

Fixed point theory in Hilbert spaces finds numerous applications in nonlinear numerical analysis and optimization \cite{bauschke2017convex,cegielski2012iterative}, which, roughly speaking, provides a unified mathematical framework for such kinds of problems. As such, a large volume of literature on the topic has emerged, including the investigation of fixed point theory itself and its applications \cite{eckstein1992douglas,attouch2010parallel,iiduka2016convergence,borwein2017convergence,haskell2017random,yi2017distributed,banjac2018tight,xu2018bregman}.

Although fruitful results have been reported on fixed point theory \cite{bauschke2017convex}, most of them are on centralized problems, that is, there is a global computing unit or coordinator who is aware of all the operators' information. Compared with centralized problems, distributed ones enjoy overwhelming advantages, such as, lower cost, higher robust to failures, and less storage, and so on \cite{ren2010distributed}. Along this line, recently, a distributed problem for finding a common fixed point of a group of paracontraction operators was studied in \cite{fullmer2018distributed,fullmer2016asynchronous}, which is motivated by a typical problem, that is, solving a linear algebraic equation in the Euclidean space in a distributed manner, where a multiple of agents hold private partial information on the linear equation and thus all agents need to cooperatively solve the problem through local communications \cite{mou2015distributed,wang2016distributed,wang2017further,wang2018distributed,alaviani2018distributed}. Meanwhile, the case with strongly quasi-nonexpansive operators was reported in \cite{liu2017distributed}. It is worthwhile to note that the aforesaid works have focused on the Euclidean space with exact knowledge of operators.

This paper aims to develop distributed algorithms for a collection of autonomous agents to seek a common fixed point of nonexpansive operators or mappings, which are privately held by individual agents, in real Hilbert spaces. Note that nonexpansive operators are more general than the operators considered in \cite{fullmer2018distributed,fullmer2016asynchronous,liu2017distributed}, and in fact they include the paracontraction operators and strongly quasi-nonexpansive operators as special cases. It is also noted that the nonexpansive operators include some celebrated operators, such as, projections, the proximal map, and the gradient descent map $x\mapsto x-\alpha\nabla f(x)$, where $f$ is a differentiable and convex function, $\nabla f$ is the gradient of $f$, being Lipschitz with constant $L$, and the constant $\alpha$ satisfies $0<\alpha<2/L$.

On the other hand, it is well known that the classical Krasnosel'ski\u{\i}-Mann (KM) iteration is a quintessential algorithm to find a fixed point for a nonexpansive operator \cite{cominetti2014rate,liang2016convergence,matsushita2017convergence,kanzow2017generalized,bravo2018sharp,bravo2018rates,shehu2018convergence}. Note that Picard iteration does not converge in general for a nonexpansive operator. The KM iteration is firstly proposed in \cite{mann1953mean,krasnosel1955two}, which have so far received tremendous attention \cite{cominetti2014rate,liang2016convergence,matsushita2017convergence,kanzow2017generalized,bravo2018sharp,bravo2018rates,shehu2018convergence}. Moreover, the KM iteration provides a unified framework for analysis of various algorithms, such as Proximal point algorithms (PPA) \cite{rockafellar1976monotone}, forward-backward splitting method (FBS) \cite{passty1979ergodic}, Peaceman-Rachford splitting (PRS) \cite{paeceman1955numerical}, Douglas-Rachford splitting (DRS) \cite{douglas1956numerical,lions1979splitting}, alternating direction method of multipliers (ADMM) \cite{gabay1975dual}, and a three-operator splitting \cite{davis2017three}. It is shown that the KM iteration converges weakly to a fixed point of a nonexpansive operator under mild conditions \cite{reich1979weak}.

With the above observations, this paper aims at developing distributed algorithms, by extending the KM iteration to the distributed scenario, for a family of autonomous agents to seek a common fixed point of a group of nonexpansive operators in real Hilbert spaces, where each operator is privately and approximately known by individual agent. In summary, the contributions of this paper can be summarized as follows.

\begin{enumerate}
  \item An algorithm, called distributed inexact Krasnosel'ski\u{\i}-Mann (D-IKM) iteration, is proposed, which, under some mild conditions, is shown to be weakly convergent to a common fixed point of the concerned nonexpansive operators. Moreover, a preliminary result on the convergence rate is provided, that is, there exists a subsequence of the sequence generated by the D-IKM iteration such that the subsequence converges at a rate $O(1/\sqrt{k})$, where $k>0$ is the iteration number. Compared with those most related works \cite{fullmer2018distributed,fullmer2016asynchronous,liu2017distributed}, all of which focus on Euclidean spaces with exact knowledge of operators and do not analyze the convergence rate, this paper considers more general spaces, i.e., real Hilbert spaces, with only approximate knowledge of operators, and also presents a result on the convergence speed.
  \item To reduce computational complexity and burden of storage and transmission, another algorithm, named distributed inexact block-coordinate Krasnosel'ski\u{\i}-Mann (D-IBKM) iteration, is developed, where only a part of coordinate is updated at each iteration for each agent. Under mild conditions, it is proved that the D-IBKM iteration converges weakly to a common fixed point of the considered operators and the similar convergence rate as in the case 1) can also be established.
  \item Under an assumption of the (bounded) linear regularity for all operators and their fixed point sets, a convergence of $O(1/k^{\ln(1/\xi)})$ for the two proposed algorithms can be established, where $\xi\in(0,1)$ is a constant and $k$ is the iteration index. Furthermore, it is shown that the same convergence rate can be maintained with a more relaxed assumption of (bounded) power regularity for the considered operators.
\end{enumerate}

The remainder of this paper is organized as follows. Section II provides some preliminary knowledge and the problem formulation, and the D-IKM iteration is developed in Section III along with its convergence rate. Subsequently, in Section IV, the D-IBKM iteration is presented along with its convergence results. The proofs of main results in last two sections are provided in Section V. Finally, Section VI concludes this paper and discusses the direction of future research.

\section{Preliminaries and Problem Statement}\label{s2}

This section provides some notations, preliminary concepts, and the problem formation.

{\em Notations:} Let $\mathcal{H}$ be a real Hilbert space with inner product $\langle\cdot,\cdot\rangle$ and associated norm $\|\cdot\|$. For an integer $n>0$, let $\mathbb{R}$, $\mathbb{R}^n$, $\mathbb{R}^{n\times n}$, and $\mathbb{N}$ represent the sets of real numbers, $n$-dimensional real vectors, $n\times n$ real matrices, and nonnegative integers, respectively. Let $[N]:=\{1,2,,\ldots,N\}$ be the index set with an integer $N>0$, and $col(z_1,\ldots,z_k)$ be the stacked column vector of $z_i\in\mathcal{H},i\in [k]$. Denote by $P_X(z)$ the projection of a point $z\in\mathcal{H}$ onto a closed and convex set $X\subset\mathcal{H}$, i.e., $P_X(z):=\mathop{\arg\min}_{x\in X}\|z-x\|$. Moreover, denote by $I$ the identity matrix of compatible dimension, $Id$ the identity operator or mapping, and $\otimes$ the Kronecker product. Let $d_X(y)$ be the distance from a point $y$ to the set $X$, i.e., $d_X(y):=\inf_{x\in X}\|y-x\|$. Let $\lfloor c\rfloor$ and $\lceil c\rceil$ be, respectively, the largest integer less than or equal to and the smallest integer greater than or equal to real number $c$. For an operator or mapping $M:\mathcal{H}\to\mathcal{H}$, denote by $Fix(M)$ the set of fixed points of $M$, i.e., $Fix(M):=\{x\in\mathcal{H}:M(x)=x\}$. Let $\rightharpoonup$ and $\to$ denote weak and strong convergence, respectively. The closed ball with center $x$ and radius $r$ is denoted by $B(x;r)$.

To proceed, let us review some fundamental concepts in operator theory \cite{bauschke2017convex}.

Let $S$ be a nonempty subset of $\mathcal{H}$, and let $T:S\to\mathcal{H}$ be an operator or mapping. Then $T$ is called {\em nonexpansive} if for all $x,y\in S$
\begin{align}
\|T(x)-T(y)\|\leq \|x-y\|,         \label{1}
\end{align}
called {\em $\alpha$-averaged} for $\alpha\in (0,1)$ if it can be written as
\begin{align}
T=(1-\alpha)Id+\alpha R,         \label{2}
\end{align}
for some nonexpansive operator $R$, called {\em firmly nonexpansive} if for all $x,y\in S$
\begin{align}
&\|T(x)-T(y)\|^2+\|(Id-T)(x)-(Id-T)(y)\|^2      \nonumber\\
&\leq \|x-y\|^2,         \label{3}
\end{align}
called {\em quasi-nonexpansive (QNE)} if for any $x\in S$ and any $y\in Fix(T)$
\begin{align}
\|T(x)-y\|\leq \|x-y\|,         \label{4}
\end{align}
and called {\em $\rho$-strongly quasi-nonexpansive ($\rho$-SQNE)} for $\rho>0$ if for all $x\in S$ and all $y\in Fix(T)$
\begin{align}
\|T(x)-y\|^2\leq \|x-y\|^2-\rho\|x-T(x)\|^2.         \label{5}
\end{align}
It is known that the set $Fix(T)$ is closed and convex if $T$ is QNE \cite{cegielski2015application}.

We are now ready to formulate the problem considered in this paper. Specifically, the goal is for a group of autonomous agents to find a common point $x$ in real Hilbert space $\mathcal{H}$ such that
\begin{align}
F_i(x)=x,~~~i\in[N]        \label{6}
\end{align}
where $F_i:\mathcal{H}\to\mathcal{H}$ is a nonexpansive operator for all $i\in[N]$. In this problem, no global coordinator, which can access all the information of $F_i$'s, is assumed to exist. Instead, $F_i$ is assumed to be approximately and locally accessible to agent $i$ in the sense that agent $i$ can receive the approximate information $F_i(x)+\epsilon_i$ for any point $x\in\mathcal{H}$, where $\epsilon_i\in\mathcal{H}$ is an error. This is more reasonable since the precise value of $F_i(x)$ is usually hard or expensive to obtain, for instance, the exact gradient of a function. The objective of this paper is to develop a distributed algorithm to solve the problem (\ref{6}) under the aforementioned scenario. One possible way to solve the problem is to generalize the classical centralized KM iteration to the distributed case. In doing so, it is helpful to briefly introduce the KM iteration.

For a nonexpansive operator $T$, a well-known method for finding a fixed point of $T$ is the so-called inexact KM iteration \cite{liang2016convergence,bravo2018rates}, that is,
\begin{align}
x_{k+1}=x_k+\alpha_k (T(x_k)+\epsilon_k-x_k),        \label{7}
\end{align}
where $\epsilon_k$ is the error of approximating $T(x_k)$, and $\{\alpha_k\}_{k\in\mathbb{N}}\in[0,1]$ is a sequence of relaxation parameters. When $\epsilon_k\equiv 0$ for all $k\in\mathbb{N}$, (\ref{7}) reduces to the classical KM iteration \cite{cominetti2014rate,matsushita2017convergence}. It has been shown that the (inexact) KM iteration converges weakly to a fixed point of $T$ under mild conditions \cite{reich1979weak,liang2016convergence,bravo2018rates}, for example, when $\sum_{j=1}^\infty \alpha_j(1-\alpha_j)=\infty$ for the KM iteration \cite{reich1979weak}.

Now, let us introduce the graph theory for describing the communication pattern among all agents \cite{ren2010distributed}. Specifically, the communication mode among $N$ agents can be modeled by a digraph $\mathcal{G}=(\mathcal{V},\mathcal{E})$, where $\mathcal{V}=[N]$ is the node or vertex set, and $\mathcal{E}\subset\mathcal{V}\times\mathcal{V}$ is the edge set. An {\em edge} $(i,j)\in\mathcal{E}$ means that agent $i$ is capable of transmitting its information to agent $j$, in which case agent $i$ is called a {\em neighbor} of agent $j$. A {\em directed path} from $i_1$ to $i_l$ is a sequence of edges of the form $(i_1,i_2),(i_3,i_4),\ldots,(i_{l-1},i_{l})$. A graph is called {\em strongly connected} if there exists at least a directed path from any node to any other node in this graph. In this paper, the communication graph for all agents is assumed to be time-varying, that is, any two agents can have different communication status at different time steps. In this case, the graph is denoted as $\mathcal{G}_k=(\mathcal{V},\mathcal{E}_k)$, where $k\in\mathbb{N}$ indicates the time index. The union of graphs $\mathcal{G}_l=(\mathcal{V},\mathcal{E}_l),l=1,\ldots,m$ is defined as $\cup_{l=1}^m\mathcal{G}_l=(\mathcal{V},\cup_{l=1}^m\mathcal{E}_l)$. At each time $k\in\mathbb{N}$, there exists an adjacency matrix $A_k=(a_{ij,k})\in\mathbb{R}^{N\times N}$ such that $a_{ij,k}>0$ if $(j,i)\in\mathcal{E}_k$, and $a_{ij,k}=0$ otherwise. Assume that $a_{ii,k}>0$ for all $i\in[N]$ and all $k\in\mathbb{N}$. For communication graphs, we have the following standard assumptions.

\begin{assumption}[Graph Connectivity and Weights Rule]\label{a1}
~
\begin{enumerate}
  \item The time-varying graphs $\mathcal{G}_k$ are uniformly jointly strongly connected, that is, there exists an integer $Q>0$ such that the graph union $\cup_{l=1}^Q \mathcal{G}_{k+l}$ is strongly connected for all $k\geq 0$.
  \item For all $k\in\mathbb{N}$, $A_k$ is row-stochastic, i.e., $\sum_{j=1}^N a_{ij,k}=1$ for all $i\in[N]$, and there exists a constant $\underline{a}\in(0,1)$ such that $a_{ij,k}>\underline{a}$ whenever $a_{ij,k}>0$.
\end{enumerate}

\end{assumption}

To end this section, it is convenient for us to list a useful lemma.

\begin{lemma}[\cite{xie2018distributed}]\label{l1}
Let Assumption \ref{a1} hold and define $A^{s:k}:=A_{s-1}\cdots A_k$ for $s\geq k$ with the convention $A^{k:k}=I$. Then, for any $k\geq 0$, there exists a vector $\pi_k=col(\pi_{1,k},\ldots,\pi_{N,k})\in\mathbb{R}^N$ such that $1_N^\top\pi_k=1$ and the following statements hold.
\begin{enumerate}
  \item $|a_{ij}^{s:k}-\pi_{j,k}|\leq \varpi\xi^{s-k}$ for all $s\geq k$ and $i,j\in[N]$, where $\varpi>0$ and $\xi\in(0,1)$ are some constants, and $a_{ij}^{s:k}$ is the $(i,j)$-th entry of $A^{s:k}$.
  \item There exists a constant $\underline{\pi}\geq \underline{a}^{Q(N-1)}$ such that $\pi_{l,k}\geq \underline{\pi}$ for all $k\geq 0$ and all $l\in[N]$.
  \item $\pi_k^\top=\pi_{k+1}^\top A_k$.
\end{enumerate}
\end{lemma}

\section{The D-IKM Iteration}\label{s3}

This section aims to solve problem (\ref{6}) by developing a distributed algorithm, called distributed inexact KM (D-IKM) iteration.

Motivated by the inexact KM iteration given in (\ref{7}), the D-IKM iteration is proposed as follows
\begin{align}
x_{i,k+1}=\hat{x}_{i,k}+\alpha_{i,k}(F_i(\hat{x}_{i,k})+\epsilon_{i,k}-\hat{x}_{i,k}),~~i\in[N]          \label{8}
\end{align}
where
\begin{align}
\hat{x}_{i,k}:=\sum_{j=1}^N a_{ij,k}x_{j,k}         \label{9}
\end{align}
represents the aggregate information received from its neighbors at time step $k$, $x_{i,k}$ is an estimate of a common fixed point of $F_i$'s by agent $i$ at time instant $k\geq 0$, $\epsilon_{i,k}$ is an error of approximating $F_i(\hat{x}_{i,k})$ by agent $i$, and $\{\alpha_{i,k}\}_{k\in\mathbb{N}}$ is a sequence of relaxation parameters for agent $i$, which is assumed to satisfy
\begin{align}
\alpha_{i,k}\in [\alpha,1-\alpha]        \label{10}
\end{align}
for some constant $\alpha\in(0,1/2]$ and for all $i\in[N], k\in\mathbb{N}$.

For the ease of exposition, let us denote by $\ell_+^1$ the set of summable sequences in $[0,+\infty)$, $X_i:=Fix(F_i)$ the set of fixed points of $F_i$, $X^*:=\cap_{i=1}^NX_i$ the set of common fixed points of all $F_i$'s which is assumed to be nonempty, and
\begin{align}
M_{i,k}:=(1-\alpha_{i,k})Id+\alpha_{i,k}F_i,~~\forall i\in[N],~k\in\mathbb{N}.         \label{11}
\end{align}

We are now ready to present the first main result as follows.

\begin{theorem}\label{t1}
For the D-IKM iteration (\ref{8}) with $\{\|\epsilon_{i,k}\|\}_{k\in\mathbb{N}}\in\ell_+^1$ for all $i\in[N]$, under Assumption \ref{a1}, the following two statements hold:
\begin{enumerate}
  \item All $x_{i,k}$'s are bounded and converge weakly to a common point in $X^*$; and
  \item There exists a subsequence $\{k_l\}_{l=1}^\infty\subset \mathbb{N}$, such that
  \begin{align}
  \|F_i(x_{i,k_l})-x_{i,k_l}\|=O(\frac{1}{\sqrt{k_l}}),~~\forall i\in[N].        \label{12}
  \end{align}
\end{enumerate}
\end{theorem}
\begin{proof}
The proof is given in Section \ref{s5.1}.
\end{proof}

\begin{remark}\label{r1}
It is worth pointing out that it is in general standard to leverage $\|T(x)-x\|$ as a measure of the convergence speed for the centralized (inexact) KM iteration, since $\|T(x)-x\|=0$ amounts to $T(x)=x$, see \cite{bravo2018sharp,bravo2018rates,cominetti2014rate,kanzow2017generalized,liang2016convergence,matsushita2017convergence}. This is why $\|F_i(x_{i,k_l})-x_{i,k_l}\|$ is employed for measuring the convergence rate of the D-IKM iteration, as shown in (\ref{12}). However, it is noted that the result in (\ref{12}) is described by a subsequence $\{k_l\}_{l=1}^\infty$ instead of $\{k\}_{k=1}^\infty$, since the D-IKM iteration involves communications over a multi-agent network unlike the case of the centralized KM iteration. It is still open whether one can obtain the result $\|F_i(x_{i,k})-x_{i,k}\|=O(1/\sqrt{k})$ as in the case of the centralized KM iteration \cite{matsushita2017convergence}.
\end{remark}


In what follows, the convergence rate of the D-IKM iteration is further discussed under some extra assumptions. It was shown in \cite{bauschke2015linear,banjac2018tight} that the centralized KM iteration is linearly convergent under the (bounded) linear regularity assumption, which is referred to as a sufficient condition for the linear convergence of averaged nonexpansive operators. It is thus natural for us to ask if the linear convergence can still be maintained for the distributed case, i.e., the D-IKM iteration, under the same assumption. To proceed, let us first review the concept of (bounded) linear regularity.

\begin{definition}[\cite{bauschke2015linear}]\label{d1}
Let $\mathcal{D}$ be a nonempty subset of $\mathcal{H}$, $T:\mathcal{D}\to\mathcal{H}$ be an operator with $Fix(T)\neq \emptyset$, and $\{S_i\}_{i\in I}$ be a finite collection of closed convex subsets of $\mathcal{H}$ with $S:=\cap_{i=I} S_i\neq \emptyset$, where $I$ is a finite index set. It is said that:
\begin{enumerate}
  \item $T$ is {\em linearly regular} with constant $\kappa\geq 0$ if for all $x\in\mathcal{D}$
  \begin{align}
  d_{Fix(T)}(x)\leq \kappa\|x-T(x)\|.           \label{13}
  \end{align}
  \item $T$ is {\em boundedly linearly regular} if for any bounded set $\Theta\subset\mathcal{D}$, there exists $\kappa\geq 0$ such that for all $x\in \Theta$
  \begin{align}
  d_{Fix(T)}(x)\leq \kappa\|x-T(x)\|.           \label{14}
  \end{align}
  \item $\{S_i\}_{i\in I}$ is {\em linearly regular} with constant $\mu>0$ if $d_S(x)\leq \mu\max_{i\in I}d_{S_i}(x)$ for all $x\in\mathcal{D}$.
  \item $\{S_i\}_{i\in I}$ is {\em boundedly linearly regular} if for any bounded set $\Theta\subset\mathcal{D}$, there exists $\mu>0$ such that $d_S(x)\leq \mu\max_{i\in I}d_{S_i}(x)$ for all $x\in \Theta$.
\end{enumerate}
\end{definition}

One example for linearly regular operators is the projection operator $P_C$ on a closed convex set $C\subset\mathcal{H}$, as it is easy to verify that $d_{Fix(P_C)}(x)=d_C(x)= \|x-P_C(x)\|$. The above notions have been thoroughly investigated in \cite{bauschke2015linear,banjac2018tight}. For instance, suppose that $I=[m]$, then $\{S_i\}_{i\in I}$ is boundedly linearly regular if $S_m\cap int(S_1\cap\cdots\cap S_{m-1})\neq \emptyset$, where $int(C)$ denotes the set of interior points of set $C$. Please refer to \cite{bauschke2015linear,banjac2018tight} for more details and \cite{dontchev2009implicit} for another relevant notion, i.e., metric (sub-)regularity for set-valued mappings.

To proceed, the assumption of the bounded linear regularity is explicitly given below.

\begin{assumption}\label{a2}
$F_i$ is boundedly linearly regular for each $i\in[N]$, and the sets $\{Fix(F_i)\}_{i\in[N]}$ are boundedly linearly regular.
\end{assumption}

In view of Theorem \ref{t1}, it is known that all $x_{i,k}$'s are bounded, say $\|x_{i,k}\|\leq \chi$ for a constant $\chi>0$ and for all $i\in[N],k\in\mathbb{N}$, which leads to that there exist constants $\kappa_0\geq 0$ and $\kappa_i\geq 0,i\in[N]$ such that for all $y\in B(0;\chi)\subset\mathcal{H}$
\begin{align}
d_{X_i}(y)&\leq \kappa_i\|F_i(y)-y\|,~~\forall i\in[N]        \label{15}\\
d_{X^*}(y)&\leq \kappa_0\max_{i\in[N]}d_{X_i}(y),          \label{16}
\end{align}
if Assumption \ref{a2} holds.

With the above preparations, we are now in a position to give the result on the D-IKM iteration's stronger convergence.

\begin{theorem}\label{t2}
Under Assumptions \ref{a1} and \ref{a2}, all $x_{i,k}$'s in the D-IKM iteration (\ref{8}) converge strongly to a common point in $X^*$, if there holds
\begin{align}
\alpha_c<\min\Big\{\frac{1}{2\kappa_c\kappa_0}\sqrt{\frac{\underline{\pi}}{2N\gamma_2}},1-\alpha\Big\},      \label{17}
\end{align}
where
\begin{align}
\gamma_2:=\frac{24N^3\varpi^2\xi^2}{(1-\xi)^2}\Big(2+\frac{1}{4N\kappa_c^2\kappa_0^2}\Big),      \label{18}
\end{align}
$\kappa_c:=\max_{i\in[N]}\kappa_i$, and $\alpha_c:=\max_{i\in[N],k\in\mathbb{N}}\alpha_{i,k}$. Moreover, in the absence of the approximate errors (called D-KM iteration for (\ref{8}) in this case), i.e., $\epsilon_{i,k}\equiv 0$ for all $i\in[N]$ and $k\in\mathbb{N}$, all $x_{i,k}$'s converge to a common point in $X^*$ at a rate of $O(1/k^{\ln(1/\xi)})$ under condition (\ref{17}), where $\xi$ is given in Lemma \ref{l1}.
\end{theorem}
\begin{proof}
The proof is given in Section \ref{s5.1}.
\end{proof}

\begin{remark}\label{r3}
From the above theorem, it can be obviously seen that the D-KM iteration enjoys the convergence rate $O(1/k^{\ln(1/\xi)})$, i.e., $O(\xi^{\ln k})$, which is slower than the linear convergence rate, i.e., $O(\xi^k)$. The slower convergence rate for the D-KM iteration can be attributed to local communications among agents, since all agents need to exchange their own information to their neighbors for achieving the synchronization of $x_{i,k}$'s for all $i\in[N]$. In this regard, it is unknown whether or not the linear convergence rate can be achieved for the D-KM iteration under the same assumptions or the bounded power regularity introduced later in Definition \ref{d2}, which is left as our future work.
\end{remark}

As a matter of fact, the convergence rate $O(1/k^{\ln(1/\xi)})$ for the D-KM iteration can still be ensured under another relaxed assumption. Specifically, we introduce a novel concept of bounded power regularity for a family of operators.

\begin{definition}\label{d2}
Let $\mathcal{D}$ be a nonempty subset of $\mathcal{H}$, and let $T_i:\mathcal{D}\to\mathcal{H}$ be an operator for each $i\in[m]$, along with $S^*:=\cap_{i=1}^m Fix(T_i)\neq \emptyset$. It is said that:
\begin{enumerate}
  \item $\{T_i\}_{i=1}^m$ are {\em power regular} with constant $\kappa\geq 0$ if for all $x\in\mathcal{D}$
  \begin{align}
  d_{S^*}(x)\leq \kappa\sum_{i=1}^m \|x-T_i(x)\|.           \label{19}
  \end{align}
  \item $\{T_i\}_{i=1}^m$ are {\em boundedly power regular} if for any bounded set $\Theta\subset\mathcal{D}$, there exists $\kappa\geq 0$ such that for all $x\in \Theta$
  \begin{align}
  d_{S^*}(x)\leq \kappa\sum_{i=1}^m\|x-T_i(x)\|.           \label{20}
  \end{align}
\end{enumerate}
\end{definition}

In the sequel, it is shown that (bounded) power regularity for a set of operators can be implied by (bounded) linear regularities of each operator and their fixed point sets.

\begin{proposition}\label{p1}
For a finite family of operators $T_i:\mathcal{D}\to\mathcal{H}$, $i\in[m]$, along with $S^*:=\cap_{i=1}^m Fix(T_i)\neq \emptyset$, if $T_i$ is (boundedly) linearly regular for each $i\in[m]$ and meanwhile the sets $\{Fix(T_i)\}_{i=1}^m$ are (boundedly) linearly regular, then $\{T_i\}_{i=1}^m$ are (boundedly) power regular.
\end{proposition}
\begin{proof}
Let us first focus on linear regularity. With reference to the conditions in this proposition, there exist constants $\kappa_i$ and $\mu$ such that for all $x\in\mathcal{D}$
\begin{align}
d_{Fix(T_i)}(x)&\leq \kappa_i\|x-T_i(x)\|,           \nonumber\\
d_{S^*}(x)&\leq \mu\max_{i\in[m]}d_{Fix(T_i)}(x),     \nonumber
\end{align}
which implies that
\begin{align}
\max_{i\in[m]}d_{Fix(T_i)}(x)\leq \sum_{i=1}^m d_{Fix(T_i)}(x)\leq \bar{\kappa}\sum_{i=1}^m\|x-T_i(x)\|,     \nonumber
\end{align}
where $\bar{\kappa}:=\max_{i\in[m]}\kappa_i$. As a result, one can obtain that $d_{S^*}(x)\leq \mu\max_{i\in[m]}d_{Fix(T_i)}(x)\leq \mu\bar{\kappa}\sum_{i=1}^m\|x-T_i(x)\|$, which thereby implies the power regularity for the set of $T_i$'s according to Definition \ref{d2}. Furthermore, the case with bounded power regularity can be similarly proved.
\end{proof}

It can be seen from Proposition \ref{p1} that (bounded) power regularity in Definition \ref{d2} is more relaxed than the notion of (bounded) linear regularity in Definition \ref{d1}. In fact, (bounded) power regularity in Definition \ref{d2} is strictly looser than (bounded) linear regularity in Definition \ref{d1}, which can be illustrated by the following example.

\begin{example}\label{e1}
Let $\mathcal{H}=\mathbb{R}$ and $\mathcal{D}=[0,1)$ in Definitions \ref{d1} and \ref{d2}, and consider two operators as $T_1(x)=x^2$ and $T_2(x)=P_C(x)$ for $x\in\mathcal{D}$, where $C=[0,1/2]$. Then, it is easy to see that $Fix(T_1)=\{0\}$, $Fix(T_2)=C$, and hence $\Omega:=Fix(T_1)\cap Fix(T_2)=\{0\}$. It is also straightforward to obtain that $d_{Fix(T_1)}(x)=x$, $\|x-T_1(x)\|=x(1-x)$ for $x\in\mathcal{D}$, thus leading to that $T_1$ is not linearly regular since $\|x-T_1(x)\|\to 0$ as $x\to 1$. But one can easily check that there holds $d_\Omega(x)\leq 2\sum_{i=1}^2\|x-T_i(x)\|$ for all $x\in\mathcal{D}$, which indicates that $\{T_i\}_{i=1}^2$ are power regular with constant $2$.
\end{example}

It is also noteworthy that notions in Definition \ref{d2} can be regarded as a generalization of (bounded) linear regularity for a single operator to multiple operators. Then, instead of Assumption \ref{a2}, the following less restrictive assumption can be made.

\begin{assumption}\label{a3}
$\{F_i\}_{i\in[N]}$ is boundedly power regular.
\end{assumption}

With this assumption, one can obtain the following result.

\begin{theorem}\label{t3}
Let Assumptions \ref{a1} and \ref{a3} hold. Then all $x_{i,k}$'s in the D-KM iteration, i.e., $\epsilon_{i,k}\equiv 0$ in (\ref{8}) for all $i\in[N]$ and $k\in\mathbb{N}$, converge to a common point in $X^*$ with a rate $O(1/k^{\ln(1/\xi)})$, if condition (\ref{17}) holds, where $\xi$ is given in Lemma \ref{l1}.
\end{theorem}
\begin{proof}
The proof is given in Section \ref{s5.1}.
\end{proof}

As seen from Theorem \ref{t3}, the convergence rate is proportional to $k^{-\ln(1/\xi)}$, like a power function of $k$, which is the reason for calling the ``power'' regularity in Definition \ref{d2}.

\section{The D-IBKM Iteration}\label{s4}

The focus of this section is on randomly updating a part of the coordinate for each agent, instead of the entire coordinate, in order to reduce the computational complexity and the burden of storage and transmission, especially for the case with large-scale coordinates and large-scale network, as investigated for centralized algorithms \cite{ortega1970iterative,combettes2004solving,combettes2015stochastic}.

To begin with, it is convenient to introduce some notations employed in this section.

{\em Notations:} $\mathcal{H}=\mathcal{H}_1\oplus\cdots\oplus\mathcal{H}_m$ is the direct Hilbert sum with Borel $\sigma$-algebra $\mathcal{B}$, where $\mathcal{H}_i,i\in[m]$ is a separable real Hilbert space, along with the same inner product $\langle\cdot,\cdot\rangle$ and associated norm $\|\cdot\|$. A $\mathcal{H}$-valued random variable is a measurable map $x:(\Omega,\mathcal{F})\to(\mathcal{H},\mathcal{B})$ with the standing probability space $(\Omega,\mathcal{F},\mathbb{P})$, endowed with the expectation $\mathbb{E}$, where a measurable (or $\mathcal{F}$-measurable) map means that there holds $\{\omega\in\Omega: x(w)\in S\}\subset \mathcal{F}$ for every set $S\in\mathcal{B}$. Let $x=(x_1,\ldots,x_m)$ denote a generic vector in $\mathcal{H}$, and let $\sigma(G)$ denote the $\sigma$-algebra generated by the collection $G$ of random variables. Denote by $\mathfrak{F}=\{\mathcal{F}_k\}_{k\in\mathbb{N}}$ a filtration, i.e., each $\mathcal{F}_i$ is a sub-sigma algebra of $\mathcal{F}$ such that $\mathcal{F}_k\subset\mathcal{F}_{k+1}$ for all $k\in\mathbb{N}$. Let $\ell_+(\mathfrak{F})$ be the set of $[0,\infty)$-valued random variable sequence $\{\zeta_k\}_{k\in\mathbb{N}}$ adapted to $\mathfrak{F}$, i.e., $\zeta_k$ is $\mathcal{F}_k$-measurable for all $k\in\mathbb{N}$, and define $\ell_+^1(\mathfrak{F})=\{\{\zeta_k\}_{k\in\mathbb{N}}\in\ell_+(\mathfrak{F}):\sum_{k\in\mathbb{N}}\zeta_k<\infty~a.s.\}$. Throughout this section, all inequalities and equalities are understood to hold $\mathbb{P}$-almost surely whenever in the presence of random variables, even though ``$\mathbb{P}$-almost surely'' is not explicitly expressed. For brevity, we abbreviate ``$\mathbb{P}$-almost surely'' as ``a.s.'' subsequently.

Consider now problem (\ref{6}). In this case, $F_i:x\mapsto (F_{il}(x))_{l\in[m]}$ is nonexpansive with $F_{il}:\mathcal{H}\to\mathcal{H}_l$ being measurable for all $i\in[N]$ and $l\in[m]$. To solve this problem, a block-coordinate based distributed algorithm, called distributed inexact block-coordinate KM (D-IBKM) iteration, is proposed as follows,
\begin{align}
x_{il,k+1}=\hat{x}_{il,k}+b_{il,k}\alpha_{i,k}(F_{il}(\hat{x}_{i,k})+\epsilon_{il,k}-\hat{x}_{il,k})     \label{01}
\end{align}
for $l\in[m]$ and $i\in[N]$, where $x_{i,k}=(x_{i1,k},\ldots,x_{im,k})$ serves as an estimate of a solution to problem (\ref{6}) for agent $i$ at time $k\geq 0$, $\hat{x}_{il,k}:=\sum_{j=1}^N a_{ij,k}x_{jl,k}$ for all $l\in[m]$, $\hat{x}_{i,k}=(\hat{x}_{i1,k},\ldots,\hat{x}_{im,k})$ is an aggregate information of agent $i$ received from its neighbors at time slot $k$, $\{b_{i,k}\}_{k\in\mathbb{N}}$ is a sequence of identically distributed $\Lambda$-valued random variables with $\Lambda:=\{0,1\}^m\backslash \{0\}$, $b_{i,k}=(b_{i1,k},\ldots,b_{im,k})$, $\epsilon_{i,k}=(\epsilon_{i1,k},\ldots,\epsilon_{im,k})$ is a $\mathcal{H}$-valued random variable, viewed as the error of approximating $F_i(\hat{x}_{i,k})$, and $\{\alpha_{i,k}\}_{k\in\mathbb{N}}$ is a sequence of relaxation parameters, satisfying $\alpha_{i,k}\in[\alpha,1-\alpha]$ for a constant $\alpha\in(0,1/2]$. Wherein, let $x_{i,0}$ be a $\mathcal{H}$-valued random variable for all $i\in[N]$.

To proceed, set $\chi_k:=\sigma(\chi_{1,k},\ldots,\chi_{N,k})$ with $\chi_{i,k}:=\sigma(x_{i,0},\ldots,x_{i,k})$, and let $\mathcal{E}_{i,k}:=\sigma(b_{i,k})$ for $i\in[N]$ and $k\in\mathbb{N}$, for which it is assumed that $\mathcal{E}_{i,k}$ is independent of $\chi_k$ and $\mathcal{E}_{j,k}$ for $j\neq i\in[N]$. Also, define $\chi=\{\chi_k\}_{k\in\mathbb{N}}$. In the meantime, assume that $p_l:=\mathbb{P}(b_{il,0}=1)>0$ for all $i\in[N]$ and $l\in[m]$, meaning that every block-coordinate has a chance to update.

Regarding iteration (\ref{01}), it can be equivalently written as
\begin{align}
x_{il,k+1}=\hat{x}_{il,k}+\alpha_{i,k}(T_{il,k}+\varepsilon_{il,k}-\hat{x}_{il,k}),      \label{02}
\end{align}
where $\varepsilon_{il,k}:=b_{il,k}\epsilon_{il,k}$, and
\begin{align}
T_{il,k}:=\hat{x}_{il,k}+b_{il,k}(F_{il}(\hat{x}_{i,k})-\hat{x}_{il,k}).            \label{03}
\end{align}

After setting $T_{i,k}:=(T_{i1,k},\ldots,T_{im,k})$ and $\varepsilon_{i,k}:=(\varepsilon_{i1,k},\ldots,\varepsilon_{im,k})$, (\ref{02}) can be compactly written as
\begin{align}
x_{i,k+1}=\hat{x}_{i,k}+\alpha_{i,k}(T_{i,k}+\varepsilon_{i,k}-\hat{x}_{i,k}).      \label{04}
\end{align}

Similarly to Section \ref{s3}, denote by $X_i:=Fix(F_i)$ the set of fixed points of $F_i$, and $X^*:=\cap_{i=1}^NX_i$ the set of common fixed points of all $F_i$'s which is assumed to be nonempty. It is also necessary to define a new norm $|||\cdot|||$ with associated inner product $\langle\langle\cdot,\cdot\rangle\rangle$ on $\mathcal{H}$ as in \cite{combettes2015stochastic}
\begin{align}
|||y|||^2&:=\sum_{l=1}^m \frac{1}{p_l}\|y_l\|^2,                 \nonumber\\
\langle\langle y,z\rangle\rangle &:=\sum_{l=1}^m \frac{1}{p_l}\langle y_l,z_l\rangle,~~~\forall y,z\in\mathcal{H}.              \label{09}
\end{align}
It is noted that $\|y\|^2\leq |||y|||^2\leq \|y\|^2/p_0$, meaning that the two norms are equivalent, where $p_0:=\min_{l\in[m]}p_l$.

Equipped with the above preparations, we are now ready to present the main result of this section.

\begin{theorem}\label{t4}
For the D-IBKM iteration (\ref{01}) under Assumption \ref{a1} and the assumption that $\sum_{k\in\mathbb{N}}\sqrt{\mathbb{E}(\|\epsilon_{i,k}\|^2|\chi_k)}<\infty$ for all $i\in[N]$, the following two statements hold:
\begin{enumerate}
  \item All $x_{i,k}$'s are bounded and converge weakly, in the space $(\mathcal{H},|||\cdot|||)$, to a common point in $X^*$ a.s.;
  \item There exists a subsequence $\{k_s\}_{s=1}^\infty\subset \mathbb{N}$, such that for all $i\in[N]$
  \begin{align}
  \mathbb{E}(\|F_i(x_{i,k_s})-x_{i,k_s}\|)=O(\frac{1}{\sqrt{k_s}}),~~\text{a.s.}        \label{05}
  \end{align}
\end{enumerate}
\end{theorem}
\begin{proof}
The proof is given in Section \ref{s5.2}.
\end{proof}

\begin{remark}\label{r4}
It should be noted that when there is only one agent in a multi-agent network, the results in Theorem \ref{t4} reduces to the centralized case \cite{combettes2015stochastic}. However, the analysis for the distributed iteration here is more complicated than that for the centralized scenario, rendering the argument in \cite{combettes2015stochastic} not directly applicable here. In addition, the convergence rate is not investigated in \cite{combettes2015stochastic}, while the convergence speed is provided here, see also Theorem \ref{t5} below.
\end{remark}

To further investigate the convergence rate of D-IBKM in (\ref{01}), let us recall Definition \ref{d2} for the bounded power regularity of a family of operators. It is known from Theorem \ref{t4} that all $x_{i,k}$'s are bounded, connoting that there is a constant $\Upsilon>0$ such that $\|x_{i,k}\|\leq \Upsilon$ for all $i\in[N],k\in\mathbb{N}$. As a consequence, under Assumption \ref{a3}, there must exist a constant $\nu>0$ such that for all $y\in B(0;\Upsilon)\subset\mathcal{H}$
\begin{align}
d_{X^*}(y)\leq\nu \sum_{i=1}^N \|F_i(y)-y\|.       \label{06}
\end{align}

Now, the stronger convergence of D-IBKM in (\ref{01}) can be given as follows.

\begin{theorem}\label{t5}
Under Assumptions \ref{a1} and \ref{a3} for the D-IBKM iteration (\ref{01}), $\lim_{k\to\infty}\mathbb{E}(|||x_{i,k}-q_k|||^2)=0$ for all $i\in[N]$ a.s., if there holds
\begin{align}
\alpha_c<\min\Big\{\frac{p_0(1-\xi)}{4N^2\varpi\xi}\sqrt{\frac{\underline{\pi}}{2(p_0^2+8N\nu^2)}},1-\alpha\Big\},      \label{07}
\end{align}
where $q_k:=\sum_{i=1}^N\pi_{i,k}P_{X^*}(x_{i,k})$, $p_0:=\min_{l\in[m]}p_l$, $\kappa_c:=\max_{i\in[N]}\kappa_i$, $\alpha_c:=\max_{i\in[N],k\in\mathbb{N}}\alpha_{i,k}$, and $\xi$ is given in Lemma \ref{l1}. Moreover, in the absence of errors (called D-BKM iteration for (\ref{01}) in this case), i.e., $\epsilon_{i,k}\equiv 0$ for all $i\in[N]$ and $k\in\mathbb{N}$, $\mathbb{E}(|||x_{i,k}-q_k|||^2)$ converges to zero with a rate $O(1/k^{\ln(1/\xi)})$ a.s. under condition (\ref{07}).
\end{theorem}
\begin{proof}
The proof is given in Section \ref{s5.2}.
\end{proof}

\section{Convergence Analysis: Proofs of Theorems \ref{t1}-\ref{t5}}\label{s5}

This section aims to provide detailed convergence analysis for the main results in the last two sections, that is, the proofs of Theorems \ref{t1}-\ref{t5}.

\subsection{Proofs of Theorems \ref{t1}-\ref{t3}}\label{s5.1}

Let us first introduce several lemmas for the subsequent use.

\begin{lemma}[\cite{nedic2015distributed}]\label{l2}
Let $\{v_k\}$ be a sequence of nonnegative scalars such that for all $k\geq 0$
\begin{align}
v_{k+1}\leq(1+b_k)v_k-u_k+c_k,      \nonumber
\end{align}
where $b_k\geq0$, $u_k\geq0$ and $c_k\geq0$ for all $k\geq0$ with $\sum_{k=1}^\infty b_k<\infty$ and $\sum_{k=1}^\infty c_k<\infty$. Then, the sequence $\{v_k\}$ converges to some $v\geq 0$ and $\sum_{k=1}^\infty u_k<\infty$.
\end{lemma}

\begin{lemma}\label{l3}
Consider $A\in\mathbb{R}^{n\times n}$ and let $B$ be a linear operator in real Hilbert space $\mathcal{H}$, then $\|A\otimes B\|\leq na_{max}\|B\|$, where $a_{max}$ is the largest entry of the matrix $A$ in the modulus sense.
\end{lemma}
\begin{proof}
For arbitrary $x=col(x_1,\ldots,x_n)$ with $x_i\in\mathcal{H}$ and $\|x\|\leq 1$, it can be concluded that
\begin{align}
\|(A\otimes B)x\|^2&=\sum_{i=1}^n\|\sum_{j=1}^na_{ij}Bx_j\|^2       \nonumber\\
&\leq \sum_{i=1}^n a_{max}^2\Big(\sum_{j=1}^n\|B\|\cdot\|x_j\|\Big)^2    \nonumber\\
&\leq n\sum_{i=1}^n a_{max}^2\sum_{j=1}^n\|B\|^2\cdot\|x_j\|^2         \nonumber\\
&\leq n^2a_{max}^2\|B\|^2,                                        \nonumber
\end{align}
where the last inequality has used the fact that $\|x\|^2=\sum_{j=1}^n\|x_j\|^2\leq 1$. Consequently, one can obtain that $\|A\otimes B\|=\sqrt{\sup_{x\in\mathcal{H}^n,\|x\|\leq 1}\|(A\otimes B)x\|^2}\leq na_{max}\|B\|$, as claimed.
\end{proof}

\begin{lemma}\label{l4}
Let $T:\mathcal{H}\to\mathcal{H}$ be a nonexpansive operator with $Fix(T)\neq\emptyset$. Then, there holds $2\langle y-z,y-T(y)\rangle\geq \|T(y)-y\|^2$ for all $y\in\mathcal{H}$ and $z\in Fix(T)$.
\end{lemma}
\begin{proof}
For any $y\in\mathcal{H}$ and $z\in Fix(T)$, it is easy to deduce that
\begin{align}
&2\langle y-z,y-T(y)\rangle      \nonumber\\
&=\|T(y)-y\|^2+\|y-z\|^2-\|T(y)-z\|^2        \nonumber\\
&\geq \|T(y)-y\|^2,          \nonumber
\end{align}
where the inequality has exploited the nonexpansive property of $T$.
\end{proof}

The following result is a fundamental result which relates $x_{i,k+1}$ to $\hat{x}_{i,k}$ via $F_i$ for each agent $i$.

\begin{lemma}\label{l5}
Consider the D-IKM iteration (\ref{8}). For all $i\in[N]$, there holds
\begin{align}
\|F_i(x_{i,k+1})-x_{i,k+1}\|\leq \|F_i(\hat{x}_{i,k})-\hat{x}_{i,k}\|+2\alpha_{i,k}\|\epsilon_{i,k}\|.      \nonumber
\end{align}
\end{lemma}
\begin{proof}
In view of (\ref{8}), it can be asserted that for $i\in[N]$
\begin{align}
&\|F_i(x_{i,k+1})-x_{i,k+1}\|          \nonumber\\
&=\|F_i(x_{i,k+1})-F_i(\hat{x}_{i,k})-(1-\alpha_{i,k})\hat{x}_{i,k}    \nonumber\\
&\hspace{0.4cm}+(1-\alpha_{i,k})F_i(\hat{x}_{i,k})-\alpha_{i,k}\epsilon_{i,k}\|     \nonumber\\
&\leq \|F_i(x_{i,k+1})-F_i(\hat{x}_{i,k})\|+(1-\alpha_{i,k})\|F_i(\hat{x}_{i,k})-\hat{x}_{i,k}\|    \nonumber\\
&\hspace{0.4cm}+\alpha_{i,k}\|\epsilon_{i,k}\|         \nonumber\\
&\leq \|x_{i,k+1}-\hat{x}_{i,k}\|+(1-\alpha_{i,k})\|F_i(\hat{x}_{i,k})-\hat{x}_{i,k}\|    \nonumber\\
&\hspace{0.4cm}+\alpha_{i,k}\|\epsilon_{i,k}\|         \nonumber\\
&\leq \|F_i(\hat{x}_{i,k})-\hat{x}_{i,k}\|+2\alpha_{i,k}\|\epsilon_{i,k}\|,       \nonumber
\end{align}
where the second inequality has made use of the nonexpansive property of $F_i$, and the last inequality has utilized the iteration (\ref{8}).
\end{proof}

With the above lemmas at hand, we are now ready to prove Theorems \ref{t1}-\ref{t3}.

{\em Proof of Theorem \ref{t1}:} Invoking (\ref{11}), the iteration (\ref{8}) can be rewritten as
\begin{align}
x_{i,k+1}=M_{i,k}(\hat{x}_{i,k})+\alpha_{i,k}\epsilon_{i,k}.        \label{t1.1}
\end{align}
Then, for any $x^*\in X^*$, which must satisfy $M_{i,k}(x^*)=x^*$ for all $i\in[N]$ and $k\in\mathbb{N}$, it can be obtained from (\ref{t1.1}) that for all $i\in[N]$
\begin{align}
\|x_{i,k+1}-x^*\|&=\|M_{i,k}(\hat{x}_{i,k})-M_{i,k}(x^*)+\alpha_{i,k}\epsilon_{i,k}\|      \nonumber\\
&\leq \|M_{i,k}(\hat{x}_{i,k})-M_{i,k}(x^*)\|+\alpha_{i,k}\|\epsilon_{i,k}\|      \nonumber\\
&\leq \|\hat{x}_{i,k}-x^*\|+\alpha_{i,k}\|\epsilon_{i,k}\|      \nonumber\\
&\leq \sum_{j=1}^N a_{ij,k}\|x_{j,k}-x^*\|+\alpha_{i,k}\|\epsilon_{i,k}\|,      \label{t1.2}
\end{align}
where the second inequality follows from the nonexpansive property of $M_{i,k}$ because $F_i$ is nonexpansive, and the last inequality is due to the convexity of $\|\cdot\|$ and $\sum_{j=1}^N a_{ij,k}=1$ for all $i\in[N]$, see Assumption \ref{a1}.

Multiplying $\pi_{i,k+1}$ on both sides of (\ref{t1.2}) and summing over $i\in[N]$ yield that
\begin{align}
\sum_{i=1}^N\pi_{i,k+1}\|x_{i,k+1}-x^*\|&\leq \sum_{j=1}^N \pi_{j,k}\|x_{j,k}-x^*\|     \nonumber\\
&\hspace{0.4cm}+\sum_{i=1}^N\pi_{i,k+1}\alpha_{i,k}\|\epsilon_{i,k}\|,      \label{t1.3}
\end{align}
where we have employed $\pi_k^\top=\pi_{k+1}^\top A_k$ in Lemma \ref{l1}. Note that $\pi_{i,k+1}\leq 1$, $\alpha_{i,k}\leq [\alpha,1-\alpha]$, and $\{\|\epsilon_{i,k}\|\}_{k\in\mathbb{N}}\in\ell_+^1$. Applying Lemma \ref{l2} results in that $\sum_{j=1}^N \pi_{j,k}\|x_{j,k}-x^*\|$ is bounded and thus so is $x_{i,k}$ for all $i\in[N]$ and $k\in\mathbb{N}$ because of $\pi_{i,k}\geq\underline{\pi}>0$ by Lemma \ref{l1}.

Subsequently, let us denote
\begin{align}
\theta_1:=\sup_{k\in\mathbb{N},i\in[N]}(2\|M_{i,k}(\hat{x}_{i,k})-x^*\|+\alpha_{i,k}\|\epsilon_{i,k}\|).     \label{t1.4}
\end{align}
Then, in view of (\ref{8}), one can obtain that
\begin{align}
&\|x_{i,k+1}-x^*\|^2       \nonumber\\
&=\|\hat{x}_{i,k}-x^*+\alpha_{i,k}(F_i(\hat{x}_{i,k})-\hat{x}_{i,k})+\alpha_{i,k}\epsilon_{i,k}\|^2      \nonumber\\
&\leq \|\hat{x}_{i,k}-x^*+\alpha_{i,k}(F_i(\hat{x}_{i,k})-\hat{x}_{i,k})\|^2+\alpha_{i,k}\theta_1\|\epsilon_{i,k}\|      \nonumber\\
&=\|\hat{x}_{i,k}-x^*\|^2+2\alpha_{i,k}\langle\hat{x}_{i,k}-x^*,F_i(\hat{x}_{i,k})-\hat{x}_{i,k}\rangle        \nonumber\\ &\hspace{0.4cm}+\alpha_{i,k}^2\|F_i(\hat{x}_{i,k})-\hat{x}_{i,k}\|^2+\alpha_{i,k}\theta_1\|\epsilon_{i,k}\|,     \label{t1.5}
\end{align}
which, together with Lemma \ref{l4} and the convexity of the norm $\|\cdot\|^2$, implies that
\begin{align}
&\|x_{i,k+1}-x^*\|^2       \nonumber\\
&\leq\sum_{j=1}^Na_{ij,k}\|x_{j,k}-x^*\|^2-\alpha_{i,k}\|F_i(\hat{x}_{i,k})-\hat{x}_{i,k}\|^2                  \nonumber\\ &\hspace{0.4cm}+\alpha_{i,k}^2\|F_i(\hat{x}_{i,k})-\hat{x}_{i,k}\|^2+\alpha_{i,k}\theta_1\|\epsilon_{i,k}\|.     \label{t1.6}
\end{align}
By multiplying $\pi_{i,k+1}$ on both sides of (\ref{t1.6}) and summing over $i\in[N]$, it can be concluded that
\begin{align}
&\sum_{i=1}^N\pi_{i,k+1}\|x_{i,k+1}-x^*\|^2       \nonumber\\
&\leq\sum_{j=1}^N\pi_{j,k}\|x_{j,k}-x^*\|^2+\theta_1\sum_{i=1}^N\pi_{i,k+1}\alpha_{i,k}\|\epsilon_{i,k}\|                  \nonumber\\ &\hspace{0.4cm}-\sum_{i=1}^N\pi_{i,k+1}\alpha_{i,k}(1-\alpha_{i,k})\|F_i(\hat{x}_{i,k})-\hat{x}_{i,k}\|^2                 \nonumber\\
&\leq\sum_{j=1}^N\pi_{j,k}\|x_{j,k}-x^*\|^2+\theta_1(1-\alpha)\sum_{i=1}^N\|\epsilon_{i,k}\|                  \nonumber\\ &\hspace{0.4cm}-\underline{\pi}\alpha(1-\alpha)\sum_{i=1}^N\|F_i(\hat{x}_{i,k})-\hat{x}_{i,k}\|^2,       \label{t1.7}
\end{align}
where we have resorted to the facts that $\alpha_{i,k}\in[\alpha,1-\alpha]$, $\alpha_{i,k}(1-\alpha_{i,k})\geq \alpha(1-\alpha)$, and $\pi_{i,k}\in[\underline{\pi},1]$ for all $i\in[N]$ and $k\in\mathbb{N}$.

Now, summing (\ref{t1.7}) over $k\in\mathbb{N}$ gives rise to
\begin{align}
&\underline{\pi}\alpha(1-\alpha)\sum_{k=0}^\infty\sum_{i=1}^N\|F_i(\hat{x}_{i,k})-\hat{x}_{i,k}\|^2          \nonumber\\
&\leq \sum_{i=1}^N\pi_{i,0}\|x_{i,0}-x^*\|^2+\theta_1(1-\alpha)\sum_{k=0}^\infty\sum_{i=1}^N\|\epsilon_{i,k}\|,        \label{t1.8}
\end{align}
which, together with $\{\|\epsilon_{i,k}\|\}_{k\in\mathbb{N}}\in\ell_+^1$, yields that
\begin{align}
\sum_{k=0}^\infty\sum_{i=1}^N\|F_i(\hat{x}_{i,k})-\hat{x}_{i,k}\|^2<\infty,        \label{t1.9}
\end{align}
further leading to
\begin{align}
\|F_i(\hat{x}_{i,k})-\hat{x}_{i,k}\|\to 0,~\text{as}~k\to\infty.        \label{t1.10}
\end{align}

With the above preparations, we are ready to prove that $x_{i,k}$'s will reach agreement for all agents $i\in[N]$. To see this, the iteration (\ref{8}) can be written in a compact form
\begin{align}
x_{k+1}=(A_k\otimes Id)x_k+\bar{\varepsilon}_k,         \label{t1.11}
\end{align}
where $x_k:=col(x_{1,k},\ldots,x_{N,k})$, $\bar{\varepsilon}_k:=col(\alpha_{1,k}(F_1(\hat{x}_{1,k})-\hat{x}_{1,k}),\ldots,\alpha_{N,k}(F_N(\hat{x}_{N,k})-\hat{x}_{N,k}))+\varepsilon_k$, and $\varepsilon_k:=col(\alpha_{1,k}\epsilon_{1,k},\ldots,\alpha_{N,k}\epsilon_{N,k})$.

Invoking (\ref{t1.10}) and $\|\epsilon_{i,k}\|\to 0$ because of $\{\|\epsilon_{i,k}\|\}_{k\in\mathbb{N}}\in\ell_+^1$, one readily obtains that $\bar{\varepsilon}_k\to 0$. With reference to (\ref{t1.11}), by applying the same arguments as that of Lemmas 3 and 4 in \cite{xie2018distributed} and using Lemma \ref{l3}, one has that
\begin{align}
\|x_{i,k}-\bar{x}_k\|\to 0,~\text{as}~k\to\infty,~~\forall i\in[N]          \label{t1.12}
\end{align}
where $\bar{x}_k:=\sum_{i=1}^N\pi_{i,k}x_{i,k}$ is viewed as a weighted average of $x_{i,k}$'s.

We next show the weak convergence of (\ref{8}). Bearing in mind that $\{\|\epsilon_{i,k}\|\}_{k\in\mathbb{N}}\in\ell_+^1$, it can be obtained by (\ref{t1.7}) and Lemma \ref{l2} that $\sum_{i=1}^N\pi_{i,k}\|x_{i,k}-x^*\|^2$ converges. In the meantime, one has that
\begin{align}
&\sum_{i=1}^N\pi_{i,k}\|x_{i,k}-x^*\|^2          \nonumber\\
&=\sum_{i=1}^N\pi_{i,k}\|x_{i,k}-\bar{x}_k+\bar{x}_k-x^*\|^2       \nonumber\\
&=\sum_{i=1}^N\pi_{i,k}\|x_{i,k}-\bar{x}_k\|^2+2\sum_{i=1}^N \pi_{i,k}\langle x_{i,k}-\bar{x}_k,\bar{x}_k-x^*\rangle     \nonumber\\
&\hspace{0.4cm}+\|\bar{x}_k-x^*\|^2,         \nonumber
\end{align}
which yields that
\begin{align}
\|\bar{x}_k-x^*\|~\text{converges},         \label{t1.13}
\end{align}
since $\sum_{i=1}^N\pi_{i,k}\|x_{i,k}-x^*\|^2$ converges, $\|x_{i,k}-\bar{x}_k\|\to 0$ (see (\ref{t1.12})), and $|\langle x_{i,k}-\bar{x}_k,\bar{x}_k-x^*\rangle|\leq \|x_{i,k}-\bar{x}_k\|\cdot\|\bar{x}_k-x^*\|\to 0$ by Cauchy-Schwarz inequality.

On the other hand, by resorting to Lemma \ref{l5} and (\ref{t1.10}) along with $\|\epsilon_{i,k}\|\to 0$ for all $i\in[N]$, one has that
\begin{align}
\|F_i(x_{i,k})-x_{i,k}\|\to 0,~~~\forall i\in[N]     \label{t1.14}
\end{align}
which, in tandem with (\ref{t1.12}), gives rise to that for all $i\in[N]$
\begin{align}
\|F_i(\bar{x}_k)-\bar{x}_k\|&\leq\|F_i(\bar{x}_k)-F_i(x_{i,k})\|+\|F_i(x_{i,k})-x_{i,k}\|     \nonumber\\
&\hspace{0.4cm}+\|x_{i,k}-\bar{x}_k\|                      \nonumber\\
&\leq 2\|x_{i,k}-\bar{x}_k\|+\|F_i(x_{i,k})-x_{i,k}\|         \nonumber\\
&\to 0,          \label{t1.15}
\end{align}
where the nonexpansiveness of $F_i$ is employed in the second inequality.

Now, for arbitrary sequential cluster point $x_c$ of $\{\bar{x}_k\}_{k\in\mathbb{N}}$, i.e., $\bar{x}_{k_l}\rightharpoonup x_c$, in view of (\ref{t1.15}), invoking Corollary 4.28 in \cite{bauschke2017convex} yields that $x_c\in Fix(F_i)$ for all $i\in[N]$, i.e., $x_c\in X^*$, Then, in light of Lemma 2.47 in \cite{bauschke2017convex} and (\ref{t1.13}), it can be asserted that $\bar{x}_k$ converges weakly to a point in $X^*$, say $\bar{x}_k\rightharpoonup x'$.

Consequently, the weak convergence of $x_{i,k}$'s to a common point in $X^*$ can be ensured once noting the fact that for all $x\in \mathcal{H}$ and all $i\in[N]$
\begin{align}
\langle x_{i,k}-x',x\rangle &=\langle x_{i,k}-\bar{x}_k,x\rangle+\langle\bar{x}_k-x',x\rangle       \nonumber\\
&\leq \|x_{i,k}-\bar{x}_k\|\cdot\|x\|+\langle\bar{x}_k-x',x\rangle              \nonumber\\
&\to 0,               \label{t1.16}
\end{align}
where the inequality has employed Cauchy-Schwarz inequality.

It remains to show the convergence rate (\ref{12}). Let us prove it by contradiction. If there are no subsequences such that (\ref{12}) holds, then there must exist $k_0\in\mathbb{N}$, $C>0$, and $i_0\in[N]$, such that for all $k\geq k_0$
\begin{align}
\|F_{i_0}(x_{i_0,k})-x_{i_0,k}\|\geq \frac{C}{\sqrt{k}}.       \label{t1.17}
\end{align}
On the other hand, in view of Lemma \ref{l5}, it follows that
\begin{align}
&\|F_i(\hat{x}_{i,k})-\hat{x}_{i,k}\|^2       \nonumber\\
&\geq \frac{\|F_i(x_{i,k+1})-x_{i,k+1}\|^2}{2}-4\alpha_{i,k}^2\|\epsilon_{i,k}\|^2        \label{t1.0}
\end{align}
by using $(a+b)^2\leq 2(a^2+b^2)$ for scalars $a,b\geq 0$, which in combination with (\ref{t1.17}) results in that
\begin{align}
&\sum_{k=0}^\infty\sum_{i=1}^N\|F_i(\hat{x}_{i,k})-\hat{x}_{i,k}\|^2            \nonumber\\
&\geq\frac{1}{2}\sum_{k=0}^\infty\sum_{i=1}^N\|F_i(x_{i,k+1})-x_{i,k+1}\|^2-4\sum_{k=0}^\infty\sum_{i=1}^N\alpha_{i,k}^2\|\epsilon_{i,k}\|^2           \nonumber\\
&\geq\frac{1}{2}\sum_{k=k_0}^\infty\|F_{i_0}(x_{i_0,k})-x_{i_0,k}\|^2-4(1-\alpha)^2\sum_{k=0}^\infty\sum_{i=1}^N\|\epsilon_{i,k}\|^2           \nonumber\\
&\geq \frac{C^2}{2}\sum_{k=k_0}^\infty \frac{1}{k}-4(1-\alpha)^2\sum_{i=1}^N\Big(\sum_{k=0}^\infty\|\epsilon_{i,k}\|\Big)^2                 \nonumber\\
&=\infty,              \label{t1.18}
\end{align}
where the last inequality has made use of (\ref{t1.17}). It is apparent that (\ref{t1.18}) contradicts (\ref{t1.9}). Therefore, one can claim that (\ref{12}) holds. This ends the proof of Theorem \ref{t1}.
\hfill\rule{2mm}{2mm}

{\em Proof of Theorem \ref{t2}:} Define $p_{i,k}=\sum_{j=1}^N a_{ij,k}P_{X^*}(x_{j,k})$, and let
\begin{align}
\theta_2&:=\sup_{k\in\mathbb{N},i\in[N]}\{2\|M_{i,k}(\hat{x}_{i,k})-p_{i,k}\|+\alpha_{i,k}\|\epsilon_{i,k}\|\},     \label{t2.1}\\
\theta_3&:=\theta_2+\sup_{k\in\mathbb{N},i\in[N]}\{4\alpha_{i,k}^2(1-\alpha_{i,k})\|\epsilon_{i,k}\|\}.             \label{t2.2}
\end{align}
Invoking (\ref{8}), it can be concluded that
\begin{align}
&d_{X^*}^2(x_{i,k+1})         \nonumber\\
&\leq\|x_{i,k+1}-p_{i,k}\|^2       \nonumber\\
&=\|\hat{x}_{i,k}-p_{i,k}+\alpha_{i,k}(F_i(\hat{x}_{i,k})-\hat{x}_{i,k})+\alpha_{i,k}\epsilon_{i,k}\|^2      \nonumber\\
&\leq \sum_{j=1}^N a_{ij,k}d_{X^*}^2(x_{j,k})-\alpha_{i,k}(1-\alpha_{i,k})\|F_i(\hat{x}_{i,k})-\hat{x}_{i,k}\|^2      \nonumber\\
&\hspace{0.4cm}+\theta_2\alpha_{i,k}\|\epsilon_{i,k}\|,      \label{t2.3}
\end{align}
where the second inequality has adopted the same reasoning as in (\ref{t1.5}) and (\ref{t1.6}). Substituting (\ref{t1.0}) in (\ref{t2.3}), one has that
\begin{align}
&d_{X^*}^2(x_{i,k+1})         \nonumber\\
&\leq \sum_{j=1}^N a_{ij,k}d_{X^*}^2(x_{j,k})-\frac{\alpha(1-\alpha)}{2}\|F_i(x_{i,k+1})-x_{i,k+1}\|^2      \nonumber\\
&\hspace{0.4cm}+\theta_3\alpha_{i,k}\|\epsilon_{i,k}\|,      \label{t2.4}
\end{align}
where we have utilized the fact that $\alpha_{i,k}(1-\alpha_{i,k})\geq \alpha(1-\alpha)$ for $\alpha_{i,k}\in[\alpha,1-\alpha]$.

By multiplying $\pi_{i,k+1}$ on both sides of (\ref{t2.4}) and summing over $i\in[N]$, one has that
\begin{align}
&\sum_{i=1}^N\pi_{i,k+1}d_{X^*}^2(x_{i,k+1})         \nonumber\\
&\leq \sum_{j=1}^N \pi_{j,k}d_{X^*}^2(x_{j,k})+\theta_3(1-\alpha)\sum_{i=1}^N\|\epsilon_{i,k}\|      \nonumber\\
&\hspace{0.4cm}-\frac{\underline{\pi}\alpha(1-\alpha)}{2}\sum_{i=1}^N\|F_i(x_{i,k+1})-x_{i,k+1}\|^2,      \label{t2.5}
\end{align}
where the facts that $\pi_k^\top=\pi_{k+1}^\top A_k$ and $\pi_{i,k}\in[\underline{\pi},1]$ in Lemma \ref{l1} have been utilized.

To proceed, it is helpful to establish a relationship between $\|F_i(x_{i,k+1})-x_{i,k+1}\|^2$ and $\|F_i(\bar{x}_{k+1})-\bar{x}_{k+1}\|$. Specifically, it can be deduced that
\begin{align}
&\|F_i(x_{i,k+1})-x_{i,k+1}\|^2         \nonumber\\
&=\|F_i(\bar{x}_{k+1})-\bar{x}_{k+1}+F_i(x_{i,k+1})-F_i(\bar{x}_{k+1})        \nonumber\\
&\hspace{0.4cm}+\bar{x}_{k+1}-x_{i,k+1}\|^2      \nonumber\\
&\geq \big(\|F_i(\bar{x}_{k+1})-\bar{x}_{k+1}\|-\|F_i(x_{i,k+1})-F_i(\bar{x}_{k+1})        \nonumber\\
&\hspace{0.4cm}+\bar{x}_{k+1}-x_{i,k+1}\|\big)^2             \nonumber\\
&\geq \frac{1}{2}\|F_i(\bar{x}_{k+1})-\bar{x}_{k+1}\|^2-\|F_i(x_{i,k+1})-F_i(\bar{x}_{k+1})        \nonumber\\
&\hspace{0.4cm}+\bar{x}_{k+1}-x_{i,k+1}\|^2,           \label{t2.6}
\end{align}
where the last inequality has used the fact that
\begin{align}
(a-b)^2\geq \frac{a^2}{2}-b^2           \label{t2.7}
\end{align}
for two scalars $a,b\geq 0$. Moreover, it is easy to get that
\begin{align}
&\|F_i(x_{i,k+1})-F_i(\bar{x}_{k+1})+\bar{x}_{k+1}-x_{i,k+1}\|^2      \nonumber\\
&\leq 2(\|F_i(x_{i,k+1})-F_i(\bar{x}_{k+1})\|^2+\|\bar{x}_{k+1}-x_{i,k+1}\|^2)      \nonumber\\
&\leq 4\|x_{i,k+1}-\bar{x}_{k+1}\|^2,              \label{t2.8}
\end{align}
where the last inequality has leveraged the nonexpansive property of $F_i$. Now, inserting (\ref{t2.8}) into (\ref{t2.6}) yields that
\begin{align}
&\|F_i(x_{i,k+1})-x_{i,k+1}\|^2         \nonumber\\
&\geq \frac{1}{2}\|F_i(\bar{x}_{k+1})-\bar{x}_{k+1}\|^2-4\|x_{i,k+1}-\bar{x}_{k+1}\|^2,           \label{t2.9}
\end{align}

At this point, turning our attention back to (\ref{t2.5}), invoking (\ref{t2.9}) leads to that
\begin{align}
&\sum_{i=1}^N\pi_{i,k+1}d_{X^*}^2(x_{i,k+1})         \nonumber\\
&\leq \sum_{j=1}^N \pi_{j,k}d_{X^*}^2(x_{j,k})+\theta_3(1-\alpha)\sum_{i=1}^N\|\epsilon_{i,k}\|      \nonumber\\
&\hspace{0.4cm}-\frac{\underline{\pi}\alpha(1-\alpha)}{4}\sum_{i=1}^N\|F_i(\bar{x}_{k+1})-\bar{x}_{k+1}\|^2       \nonumber\\
&\hspace{0.4cm}+2\underline{\pi}\alpha(1-\alpha)\sum_{i=1}^N\|x_{i,k+1}-\bar{x}_{k+1}\|^2.      \label{t2.10}
\end{align}

Consider the term $\sum_{i=1}^N\|F_i(\bar{x}_{k+1})-\bar{x}_{k+1}\|^2$ in (\ref{t2.10}), in light of (\ref{15}), one can obtain that
\begin{align}
\sum_{i=1}^N\|F_i(\bar{x}_{k+1})-\bar{x}_{k+1}\|^2&\geq \frac{1}{\kappa_c^2}\sum_{i=1}^N d_{X_i}^2(\bar{x}_{k+1})         \nonumber\\
&\geq \frac{1}{\kappa_c^2}\max_{i\in[N]} d_{X_i}^2(\bar{x}_{k+1})       \nonumber\\
&\geq \frac{1}{\kappa_c^2\kappa_0^2} d_{X^*}^2(\bar{x}_{k+1}),         \label{t2.11}
\end{align}
where the last inequality is due to (\ref{16}). Consider further the term $d_{X^*}^2(\bar{x}_{k+1})$ in (\ref{t2.11}), one has that
\begin{align}
d_{X^*}^2(\bar{x}_{k+1})&=\|\bar{x}_{k+1}-P_{X^*}(\bar{x}_{k+1})\|^2       \nonumber\\
&=\|x_{i,k+1}-P_{X^*}(\bar{x}_{k+1})+\bar{x}_{k+1}-x_{i,k+1}\|^2          \nonumber\\
&\geq \big(\|x_{i,k+1}-P_{X^*}(\bar{x}_{k+1})\|-\|\bar{x}_{k+1}-x_{i,k+1}\|\big)^2      \nonumber\\
&\geq \frac{1}{2}d_{X^*}^2(x_{i,k+1})-\|x_{i,k+1}-\bar{x}_{k+1}\|^2,         \label{t2.12}
\end{align}
where the last inequality has leveraged (\ref{t2.7}). Summing over $i\in[N]$ for (\ref{t2.12}) yields that
\begin{align}
&d_{X^*}^2(\bar{x}_{k+1})   \nonumber\\
&\geq \frac{1}{2N}\sum_{i=1}^Nd_{X^*}^2(x_{i,k+1})-\frac{1}{N}\sum_{i=1}^N\|x_{i,k+1}-\bar{x}_{k+1}\|^2.         \label{t2.13}
\end{align}

Substituting (\ref{t2.11}) and (\ref{t2.13}) into (\ref{t2.10}) gives rise to that
\begin{align}
&\sum_{i=1}^N\pi_{i,k+1}d_{X^*}^2(x_{i,k+1})         \nonumber\\
&\leq \sum_{i=1}^N \pi_{i,k}d_{X^*}^2(x_{i,k})-\frac{\underline{\pi}\alpha(1-\alpha)}{8N\kappa_c^2\kappa_0^2}\sum_{i=1}^N d_{X^*}^2(x_{i,k+1})      \nonumber\\
&\hspace{0.4cm}+\gamma_1\sum_{i=1}^N\|x_{i,k+1}-\bar{x}_{k+1}\|^2       \nonumber\\
&\hspace{0.4cm}+\theta_3(1-\alpha)\sum_{i=1}^N\|\epsilon_{i,k}\|,      \label{t2.14}
\end{align}
where
\begin{align}
\gamma_1:=\underline{\pi}\alpha(1-\alpha)\Big(2+\frac{1}{4N\kappa_c^2\kappa_0^2}\Big).            \label{t2.15}
\end{align}

For notation simplicity, let
\begin{align}
d_k^2&:=\sum_{i=1}^N \pi_{i,k}d_{X^*}^2(x_{i,k}),~~~\forall k\in\mathbb{N}        \label{t2.16}\\
\beta&:=1+\frac{\underline{\pi}\alpha(1-\alpha)}{8N\kappa_c^2\kappa_0^2}.          \label{t2.17}
\end{align}

Then, (\ref{t2.14}) can be written as
\begin{align}
\beta d_{k+1}^2&\leq d_k^2+\gamma_1\sum_{i=1}^N\|x_{i,k+1}-\bar{x}_{k+1}\|^2       \nonumber\\
&\hspace{0.4cm}+\theta_3(1-\alpha)\sum_{i=1}^N\|\epsilon_{i,k}\|.      \label{t2.18}
\end{align}

Consider now the term $\gamma_1\sum_{i=1}^N\|x_{i,k+1}-\bar{x}_{k+1}\|^2$. Recalling $p_{i,k}:=\sum_{j=1}^Na_{ij,k}P_{X^*}(x_{j,k})$, one can conclude that
\begin{align}
\|F_i(\hat{x}_{i,k})-\hat{x}_{i,k}\|^2&\leq \|F_i(\hat{x}_{i,k})-p_{i,k}+p_{i,k}-\hat{x}_{i,k}\|^2         \nonumber\\
&\leq 2\|F_i(\hat{x}_{i,k})-p_{i,k}\|^2+2\|\hat{x}_{i,k}-p_{i,k}\|^2          \nonumber\\
&\leq 4\|\hat{x}_{i,k}-p_{i,k}\|^2         \nonumber\\
&\leq 4\sum_{j=1}^N a_{ij,k} d_{X^*}^2(x_{j,k}),           \label{t2.19}
\end{align}
where the third inequality has employed the fact that $F_i(p_{i,k})=p_{i,k}$ and $F_i$ is nonexpansive, and the last inequality has used the convexity of $\|\cdot\|^2$. Subsequently, by multiplying $\pi_{i,k+1}$ on both sides of (\ref{t2.19}) and summing over $i\in[N]$, it follows that
\begin{align}
\sum_{i=1}^N \pi_{i,k+1}\|F_i(\hat{x}_{i,k})-\hat{x}_{i,k}\|^2\leq 4\sum_{j=1}^N \pi_{j,k}d_{X^*}^2(x_{j,k}),        \label{t2.20}
\end{align}
where $\pi_k^\top=\pi_{k+1}^\top A_k$ in Lemma \ref{l1} has been applied in the inequality. Combining (\ref{t2.20}) with the fact that $\sum_{i=1}^N \pi_{i,k+1}\|F_i(\hat{x}_{i,k})-\hat{x}_{i,k}\|^2\geq \underline{\pi}\sum_{i=1}^N\|F_i(\hat{x}_{i,k})-\hat{x}_{i,k}\|^2$ implies that
\begin{align}
\sum_{i=1}^N \|F_i(\hat{x}_{i,k})-\hat{x}_{i,k}\|^2\leq \frac{4}{\underline{\pi}}d_k^2.         \label{t2.21}
\end{align}

Bearing in mind the definition of $\bar{\varepsilon}_k$ in (\ref{t1.11}), it follows from (\ref{t2.21}) that
\begin{align}
\|\bar{\varepsilon}_k\|^2\leq\frac{8\alpha_c^2}{\underline{\pi}}d_k^2+2\|\varepsilon_k\|^2,       \label{t2.22}
\end{align}
where $\alpha_c=\max_{i\in[N],k\in\mathbb{N}}\alpha_{i,k}$. In view of (\ref{t2.22}), following the same arguments as that of Lemmas 3 and 4 in \cite{xie2018distributed} for (\ref{t1.11}), one can conclude that
\begin{align}
&\|x_{i,k+1}-\bar{x}_{k+1}\|        \nonumber\\
&\leq N\varpi\xi^{k+1}\|x_0-\bar{x}_0\|+\frac{N\varpi\xi\xi^{\lceil \frac{k+1}{2}\rceil}}{1-\xi}\sup_{l\in\mathbb{N}}\|\bar{\varepsilon}_l\|          \nonumber\\
&\hspace{0.4cm}+\frac{N\varpi\xi}{1-\xi}\sup_{\lfloor\frac{k+1}{2}\rfloor\leq l\leq k}\|\bar{\varepsilon}_l\|,        \label{t2.23}
\end{align}
which further gives rise to
\begin{align}
&\|x_{i,k+1}-\bar{x}_{k+1}\|^2      \nonumber\\
&\leq\frac{3N^2\varpi^2\xi^{k+3}}{(1-\xi)^2}\sup_{l\in\mathbb{N}}\|\bar{\varepsilon}_l\|^2+\frac{3N^2\varpi^2\xi^2}{(1-\xi)^2}\sup_{\lfloor\frac{k+1}{2}\rfloor\leq l\leq k}\|\bar{\varepsilon}_l\|^2   \nonumber\\
&\hspace{0.4cm} +3N^2\varpi^2\xi^{2(k+1)}\|x_0-\bar{x}_0\|^2        \nonumber\\
&\leq \frac{3N^2\varpi^2\xi^{k+3}}{(1-\xi)^2}\sup_{l\in\mathbb{N}}\|\bar{\varepsilon}_l\|^2+\frac{24N^2\alpha_c^2\varpi^2\xi^2}{\underline{\pi}(1-\xi)^2}\sup_{\lfloor\frac{k+1}{2}\rfloor\leq l\leq k}d_l^2       \nonumber\\
&\hspace{0.4cm}+\frac{6N^2\varpi^2\xi^2}{(1-\xi)^2}\sup_{\lfloor\frac{k+1}{2}\rfloor\leq l\leq k}\|\varepsilon_l\|^2        \nonumber\\
&\hspace{0.4cm} +3N^2\varpi^2\xi^{2(k+1)}\|x_0-\bar{x}_0\|^2,        \label{t2.24}
\end{align}
where (\ref{t2.22}) has been utilized in the second inequality.

Inserting (\ref{t2.24}) into (\ref{t2.18}), it can be then obtained that
\begin{align}
\beta d_{k+1}^2\leq d_k^2 +\gamma_2\alpha(1-\alpha)\alpha_c^2\sup_{\lfloor\frac{k+1}{2}\rfloor\leq l\leq k}d_l^2+\beta e_k,      \label{t2.25}
\end{align}
where $\gamma_2$ is defined in (\ref{18}) and
\begin{align}
e_k&:=\frac{1}{\beta}\bigg(3\gamma_1N^3\varpi^2\xi^{2(k+1)}\|x_0-\bar{x}_0\|^2         \nonumber\\
&\hspace{1.1cm}+\frac{3\gamma_1N^3\varpi^2\xi^{k+3}}{(1-\xi)^2}\sup_{l\in\mathbb{N}}\|\bar{\varepsilon}_l\|^2     \nonumber\\
&\hspace{1.1cm}+\frac{6\gamma_1N^3\varpi^2\xi^2}{(1-\xi)^2}\sup_{\lfloor\frac{k+1}{2}\rfloor\leq l\leq k}\|\varepsilon_l\|^2\bigg).       \label{t2.26}
\end{align}

Because of $d_k^2\leq \sup_{\lfloor\frac{k+1}{2}\rfloor\leq l\leq k}d_l^2$, it follows from (\ref{t2.25}) that
\begin{align}
\beta d_{k+1}^2\leq [1+\gamma_2\alpha(1-\alpha)\alpha_c^2]\sup_{\lfloor\frac{k+1}{2}\rfloor\leq l\leq k}d_l^2+\beta e_k,      \label{t2.27}
\end{align}
further implying that
\begin{align}
d_{k+1}^2\leq \gamma\sup_{\lfloor\frac{k+1}{2}\rfloor\leq l\leq k}d_l^2+e_k,      \label{t2.28}
\end{align}
where
\begin{align}
\gamma:=\frac{1+\gamma_2\alpha(1-\alpha)\alpha_c^2}{\beta}.           \label{t2.29}
\end{align}
It is easy to verify that $\gamma<1$ under condition (\ref{17}). Note that there exists $m\in\mathbb{N}\backslash\{0\}$ such that
\begin{align}
k+1\in[2^{m-1},2^m-1].        \label{o1}
\end{align}
Then, by iteratively applying (\ref{t2.28}), one can conclude that
\begin{align}
d_{k+1}^2\leq \gamma^m d_0^2+\sum_{j=0}^{m-1}\gamma^j e_{k_j-1},     \label{t2.30}
\end{align}
where
\begin{align}
\lfloor\frac{k_j}{2}\rfloor\leq k_{j+1}\leq k_j-1,~j=1,\ldots,m-1       \label{o2}
\end{align}
with $k_m=0$ and $k_0=k+1$.

Meanwhile, it can be obtained that
\begin{align}
&\sum_{j=0}^{m-1}\gamma^j e_{k_j-1}       \nonumber\\
&=\sum_{j=0}^{\lfloor \frac{m-1}{2}\rfloor}\gamma^j e_{k_j-1}+\sum_{j=\lfloor \frac{m-1}{2}\rfloor+1}^{m-1}\gamma^j e_{k_j-1}     \nonumber\\
&=\sum_{j=0}^{\lfloor \frac{m-1}{2}\rfloor}\gamma^j e_{k_j-1}+\gamma^{\lfloor \frac{m-1}{2}\rfloor}\sum_{j=1}^{\lceil \frac{m-1}{2}\rceil}\gamma^j e_{k_{\lfloor \frac{m-1}{2}\rfloor+j}-1}   \nonumber\\
&\leq \frac{1}{1-\gamma}\sup_{l\geq k_{\lfloor \frac{m-1}{2}\rfloor}}e_{l-1}+\gamma^{\lfloor \frac{m-1}{2}\rfloor}\frac{\gamma}{1-\gamma}\sup_{l\geq 0} e_l,    \label{t2.31}
\end{align}
which, together with (\ref{t2.30}), yields that
\begin{align}
d_{k+1}^2&\leq \gamma^m d_0^2+\frac{1}{1-\gamma}\sup_{l\geq k_{\lfloor \frac{m-1}{2}\rfloor}}e_{l-1}     \nonumber\\
&\hspace{0.4cm}+\frac{\gamma\gamma^{\lfloor \frac{m-1}{2}\rfloor}}{1-\gamma}\sup_{l\geq 0} e_l.     \label{t2.32}
\end{align}

It is easy to see that $\|\varepsilon_l\|\to 0$ as $l\to\infty$ since so is $\|\epsilon_{i,l}\|$ due to $\{\epsilon_{i,l}\}_{l\in\mathbb{N}}\in\ell_+^1$ for all $i\in[N]$, and thus $e_l\to 0$ as $l\to \infty$. Moreover, it can be obtained from (\ref{o1}) that
\begin{align}
\log_2 (k+1)\leq m\leq \log_2(k+1)+1,         \label{o3}
\end{align}
which further implies that
\begin{align}
\frac{\log_2 (k+1)}{2}-\frac{3}{2}\leq \lfloor\frac{m-1}{2}\rfloor\leq \frac{\log_2(k+1)}{2}.         \label{o4}
\end{align}

On the other hand, invoking (\ref{o2}) yields that
\begin{align}
\frac{k+1}{2^{\lfloor\frac{m-1}{2}\rfloor}}-\sum_{l=1}^{\lfloor\frac{m-1}{2}\rfloor}\frac{1}{2^l}\leq k_{\lfloor\frac{m-1}{2}\rfloor}\leq k+1-\lfloor\frac{m-1}{2}\rfloor,    \nonumber
\end{align}
which, together with (\ref{o4}), leads to
\begin{align}
\sqrt{k+1}-1\leq k_{\lfloor\frac{m-1}{2}\rfloor}\leq k+1-\frac{\sqrt{k+1}}{2\sqrt{2}},    \label{o5}
\end{align}

By combining (\ref{t2.32})-(\ref{o5}), one can conclude that $d_k^2$ and thus $\|x_{i,k}-\bar{x}_k\|^2$ for all $i\in[N]$ (see (\ref{t2.24})) converge strongly to the origin, and converge at a rate of $O(\xi^{\ln k})$, i.e., $O(1/k^{\ln(1/\xi)})$, when $\epsilon_{i,k}\equiv 0$ for all $i\in[N]$ and $k\in \mathbb{N}$.

Finally, let
\begin{align}
q_k:=\sum_{i=1}^N\pi_{i,k}P_{X^*}(x_{i,k}).      \label{t2.33}
\end{align}
Then, applying the convexity of $\|\cdot\|^2$, it can be concluded that $\|\bar{x}_k -q_k\|^2\leq \sum_{i=1}^N\pi_{i,k}\|x_{i,k}-P_{X^*}(x_{i,k})\|^2=d_k^2$. Meanwhile, note that $\|x_{i,k}-q_k\|^2\leq 2\|x_{i,k}-\bar{x}_k\|^2+2\|\bar{x}_k-q_k\|^2$ and $q_k\in X^*$ for all $k\in\mathbb{N}$. Combining the above analysis completes the proof.
\hfill\rule{2mm}{2mm}

{\em Proof of Theorem \ref{t3}:} By Theorem \ref{t1}, it is known that all $x_{i,k}$'s are bounded. Therefore, according to the bounded power regularity of $\{F_i\}_{i\in[N]}$, one has that there exists a constant $\kappa_d\geq 0$ such that
\begin{align}
d_{X^*}(\bar{x}_{k+1})\leq \kappa_d\sum_{i=1}^N\|F_i(\bar{x}_{k+1})-\bar{x}_{k+1}\|,      \label{t3.1}
\end{align}
which leads to that
\begin{align}
d_{X^*}^2(\bar{x}_{k+1})\leq N\kappa_d^2\sum_{i=1}^N\|F_i(\bar{x}_{k+1})-\bar{x}_{k+1}\|^2.      \label{t3.2}
\end{align}

Note that (\ref{t3.2}) is consistent with (\ref{t2.11}) with different coefficients. Hence, following the same argument as that of Theorem \ref{t2}, the conclusions of this theorem can be asserted. The proof is thus completed.
\hfill\rule{2mm}{2mm}

\subsection{Proofs of Theorems \ref{t4} and \ref{t5}}\label{s5.2}

Let us first introduce several lemmas.

\begin{lemma}[\cite{robbins1971martin}]\label{l6}
Let $\mathfrak{F}=\{\mathcal{F}_k\}_{k\in\mathbb{N}}$ be a filtration. If $\{z_k\}_{k\in\mathbb{N}}\in\ell_+(\mathfrak{F})$, $\{\varsigma_k\}_{k\in\mathbb{N}}\in\ell_+^1(\mathfrak{F})$, $\{\vartheta_k\}_{k\in\mathbb{N}}\in\ell_+(\mathfrak{F})$, and $\{\eta_k\}_{k\in\mathbb{N}}\in\ell_+^1(\mathfrak{F})$ satisfy the following inequality a.s.:
\begin{align}
\mathbb{E}(z_{k+1}|\mathcal{F}_k)\leq (1+\varsigma_k)\|z_k\|-\vartheta_k+\eta_k,~~~\forall k\in\mathbb{N}       \nonumber
\end{align}
then, $\{\vartheta_k\}_{k\in\mathbb{N}}\in\ell_+^1(\mathfrak{F})$ and $z_k$ converges to a $[0,\infty)$-valued random variable a.s.
\end{lemma}

\begin{lemma}[\cite{bauschke2017convex}]\label{l7}
Let $x,y\in\mathcal{H}$, and let $r\in\mathbb{R}$. Then
\begin{align}
\|rx+(1-r)y\|^2&=r\|x\|^2+(1-r)\|y\|^2      \nonumber\\
&\hspace{0.4cm}-r(1-r)\|x-y\|^2.       \label{08}
\end{align}
\end{lemma}

The relationship between $x_{i,k+1}$ and $\hat{x}_{i,k}$ is revealed through $F_i$ in the following lemma.

\begin{lemma}\label{l8}
Consider the D-IBKM iteration (\ref{01}). For all $i\in[N]$, there holds
\begin{align}
\mathbb{E}(|||F_i(x_{i,k+1})-x_{i,k+1}|||^2|\chi_k)&\leq 4|||F_i(\hat{x}_{i,k})-\hat{x}_{i,k}|||^2        \nonumber\\
&\hspace{0.4cm}+16\alpha_{i,k}^2\mathbb{E}(|||\epsilon_{i,k}|||^2|\chi_k).      \nonumber
\end{align}
\end{lemma}
\begin{proof}
It follows that
\begin{align}
&|||F_i(x_{i,k+1})-x_{i,k+1}|||          \nonumber\\
&= |||F_i(x_{i,k+1})-F_i(\hat{x}_{i,k})+(1-\alpha_{i,k})(F_i(\hat{x}_{i,k})-\hat{x}_{i,k})         \nonumber\\
&\hspace{0.4cm}+\alpha_{i,k}(F_i(\hat{x}_{i,k})-T_{i,k})-\alpha_{i,k}\varepsilon_{i,k}|||     \nonumber\\
&\leq |||F_i(x_{i,k+1})-F_i(\hat{x}_{i,k})|||+(1-\alpha_{i,k})|||F_i(\hat{x}_{i,k})-\hat{x}_{i,k}|||         \nonumber\\
&\hspace{0.4cm}+\alpha_{i,k}|||F_i(\hat{x}_{i,k})-T_{i,k}|||+\alpha_{i,k}|||\varepsilon_{i,k}|||     \nonumber\\
&\leq |||x_{i,k+1}-\hat{x}_{i,k}|||+(1-\alpha_{i,k})|||F_i(\hat{x}_{i,k})-\hat{x}_{i,k}|||         \nonumber\\
&\hspace{0.4cm}+\alpha_{i,k}|||F_i(\hat{x}_{i,k})-T_{i,k}|||+\alpha_{i,k}|||\varepsilon_{i,k}|||     \nonumber\\
&\leq \alpha_{i,k}|||T_{i,k}-\hat{x}_{i,k}|||+(1-\alpha_{i,k})|||F_i(\hat{x}_{i,k})-\hat{x}_{i,k}|||         \nonumber\\
&\hspace{0.4cm}+\alpha_{i,k}|||F_i(\hat{x}_{i,k})-T_{i,k}|||+2\alpha_{i,k}|||\varepsilon_{i,k}|||,     \label{t4.x}
\end{align}
where (\ref{04}) has been employed in the equality and last inequality, and the nonexpansiveness of $F_i$ deduces the second inequality.

To proceed, let us analyze $\mathbb{E}(|||T_{i,k}-\hat{x}_{i,k}|||^2|\chi_k)$ and $\mathbb{E}(|||F_i(\hat{x}_{i,k})-T_{i,k}|||^2|\chi_k)$. In doing so, for all $i\in[N]$, $k\in\mathbb{N}$, and $l\in[m]$, define
\begin{align}
q_{il,k}(y,b)=\|y_l-F_{il}(y)+b(F_{il}(y)-y_l)\|^2       \label{t4.1}
\end{align}
for $y\in\mathcal{H},b\in\{0,1\}$. It is easy to see that $q_{il,k}(\hat{x}_{i,k},b_{il,k})$ is $\chi_k$-measurable since $F_{il}$ is so.

As a result, one can obtain that for all $i\in[N]$
\begin{align}
&\mathbb{E}(|||T_{i,k}-F_i(\hat{x}_{i,k})|||^2|\chi_k)        \nonumber\\
&=\sum_{l=1}^m\frac{1}{p_l}\mathbb{E}(\|T_{il,k}-F_{il}(\hat{x}_{i,k})\|^2|\chi_k)       \nonumber\\
&=\sum_{l=1}^m\frac{1}{p_l}\sum_{b\in\{0,1\}}\mathbb{P}(b_{il,k}=b)q_{il,k}(\hat{x}_{i,k},b_{il,k})           \nonumber\\
&=\sum_{l=1}^m\frac{1}{p_l}\mathbb{P}(b_{il,k}=0)\|\hat{x}_{i,k}-F_{il}(\hat{x}_{i,k})\|^2       \nonumber\\
&=|||F_i(\hat{x}_{i,k})-\hat{x}_{i,k}|||^2-\|F_i(\hat{x}_{i,k})-\hat{x}_{i,k}\|^2,             \label{t4.a}
\end{align}
where the third equality has used the definition (\ref{03}). Similarly, one has that
\begin{align}
\mathbb{E}(|||T_{i,k}-\hat{x}_{i,k}|||^2|\chi_k)=\|F_i(\hat{x}_{i,k})-\hat{x}_{i,k}\|^2.             \label{t4.b}
\end{align}

Now, by squaring (\ref{t4.x}), taking the conditional expectation, and inserting (\ref{t4.a}) and (\ref{t4.b}) into it, one obtains that
\begin{align}
&\mathbb{E}(|||F_i(x_{i,k+1})-x_{i,k+1}|||^2|\chi_k)          \nonumber\\
&\leq 4[\alpha_{i,k}^2+(1-\alpha_{i,k})^2]|||F_i(\hat{x}_{i,k})-\hat{x}_{i,k}|||^2        \nonumber\\
&\hspace{0.4cm}+16\alpha_{i,k}^2\mathbb{E}(|||\varepsilon_{i,k}|||^2|\chi_k)     \nonumber\\
&\leq 4|||F_i(\hat{x}_{i,k})-\hat{x}_{i,k}|||^2+16\alpha_{i,k}^2\mathbb{E}(|||\varepsilon_{i,k}|||^2|\chi_k),     \nonumber
\end{align}
which is as claimed.
\end{proof}

With the above lemmas at hand, we are now ready to prove Theorems \ref{t4} and \ref{t5} as follows.

{\em Proof of Theorem \ref{t4}:} Throughout this subsection, let $x^*\in X^*$. Similar to (\ref{t4.a}), one can obtain that for all $i\in[N]$
\begin{align}
&\mathbb{E}(|||T_{i,k}-x^*|||^2|\chi_k)        \nonumber\\
&=\sum_{l=1}^m\frac{1}{p_l}\mathbb{P}(b_{il,k}=1)\|F_{il}(\hat{x}_{i,k})-x_l^*\|^2             \nonumber\\
&\hspace{0.4cm}+\sum_{l=1}^m\frac{1}{p_l}\mathbb{P}(b_{il,k}=0)\|\hat{x}_{i,k}-x_l^*\|^2       \nonumber\\
&=\sum_{l=1}^m\|F_{il}(\hat{x}_{i,k})-x_l^*\|^2+\sum_{l=1}^m\frac{1-p_l}{p_l}\|\hat{x}_{i,k}-x_l^*\|^2       \nonumber\\
&=|||\hat{x}_{i,k}-x^*|||^2+\|F_i(\hat{x}_{i,k})-x^*\|^2-\|\hat{x}_{i,k}-x^*\|^2           \nonumber\\
&\leq |||\hat{x}_{i,k}-x^*|||^2,             \label{t4.2}
\end{align}
where the first equality has used the definition (\ref{03}), and the nonexpansive property of $F_i$ has been applied to get the inequality.

Invoking Jensen's inequality and (\ref{t4.2}), it can be concluded that
\begin{align}
\mathbb{E}(|||T_{i,k}-x^*|||~|\chi_k)&\leq \sqrt{\mathbb{E}(|||T_{i,k}-x^*|||^2|\chi_k)}       \nonumber\\
&\leq |||\hat{x}_{i,k}-x^*|||.         \label{t4.3}
\end{align}

Then, in view of (\ref{04}), one has that
\begin{align}
&|||x_{i,k+1}-x^*|||        \nonumber\\
&=|||(1-\alpha_{i,k})(\hat{x}_{i,k}-x^*)+\alpha_{i,k}(T_{i,k}-x^*)+\alpha_{i,k}\varepsilon_{i,k}|||     \nonumber\\
&\leq (1-\alpha_{i,k})|||\hat{x}_{i,k}-x^*|||+\alpha_{i,k}|||T_{i,k}-x^*|||+\alpha_{i,k}|||\varepsilon_{i,k}|||.    \nonumber
\end{align}
Taking the conditional expectation on the above inequality yields that
\begin{align}
&\mathbb{E}(|||x_{i,k+1}-x^*|||~|\chi_k)       \nonumber\\
&\leq (1-\alpha_{i,k})|||\hat{x}_{i,k}-x^*|||+\alpha_{i,k}\mathbb{E}(|||T_{i,k}-x^*|||~|\chi_k)     \nonumber\\
&\hspace{0.4cm}+\alpha_{i,k}\mathbb{E}(|||\varepsilon_{i,k}|||~|\chi_k)    \nonumber\\
&\leq |||\hat{x}_{i,k}-x^*|||+\alpha_{i,k}\mathbb{E}(|||\varepsilon_{i,k}|||~|\chi_k)      \nonumber\\
&\leq \sum_{j=1}^N a_{ij,k}|||x_{j,k}-x^*|||+\alpha_{i,k}\mathbb{E}(|||\varepsilon_{i,k}|||~|\chi_k),      \nonumber
\end{align}
where the second inequality has exploited (\ref{t4.3}) and the last inequality has applied the convexity of norm $|||\cdot|||$. By multiplying $\pi_{i,k+1}$ on both sides of the above inequality and summing over $i\in[N]$, it is easy to obtain that
\begin{align}
&\mathbb{E}(\sum_{i=1}^N\pi_{i,k+1}|||x_{i,k+1}-x^*|||~|\chi_k)          \nonumber\\
&\leq \sum_{j=1}^N \pi_{j,k}|||x_{j,k}-x^*|||+\sum_{i=1}^N\alpha_{i,k}\mathbb{E}(|||\varepsilon_{i,k}|||~|\chi_k),      \label{t4.4}
\end{align}
where we have employed $\pi_k^\top=\pi_{k+1}^\top A_k$ and $\pi_{i,k}\leq 1$ in Lemma \ref{l1}.

By the assumption in Theorem \ref{t4}, it is straightforward to verify that
\begin{align}
\sum_{k\in\mathbb{N}}\mathbb{E}(|||\varepsilon_{i,k}|||~|\chi_k)
&\leq\sum_{k\in\mathbb{N}}\sqrt{\mathbb{E}(|||\varepsilon_{i,k}|||^2|\chi_k)}      \nonumber\\
&\leq\sum_{k\in\mathbb{N}}\sqrt{\mathbb{E}(|||\epsilon_{i,k}|||^2|\chi_k)}        \nonumber\\
&\leq \frac{1}{\sqrt{p_0}}\sum_{k\in\mathbb{N}}\sqrt{\mathbb{E}(\|\epsilon_{i,k}\|^2|\chi_k)}                \nonumber\\
&<\infty.                             \label{t4.5}
\end{align}

Now, applying Lemma \ref{l6} to (\ref{t4.4}), one can readily obtain that $\sum_{j=1}^N \pi_{j,k}|||x_{j,k}-x^*|||$ and thereby $x_{i,k}$'s are bounded a.s.

Since $x_{i,k}$ is bounded a.s., there exists $\tau_1\in(0,\infty)$ such that for all $k\in\mathbb{N},i\in[N]$
\begin{align}
\tau_1&\geq 2|||(1-\alpha_{i,k})(\hat{x}_{i,k}-x^*)+\alpha_{i,k}(T_{i,k}-x^*)|||     \nonumber\\
&\hspace{1.8cm}+\alpha_{i,k}|||\varepsilon_{i,k}|||,~~a.s.           \nonumber
\end{align}

Then, it follows that
\begin{align}
&|||x_{i,k+1}-x^*|||^2        \nonumber\\
&=|||(1-\alpha_{i,k})(\hat{x}_{i,k}-x^*)+\alpha_{i,k}(T_{i,k}-x^*)+\alpha_{i,k}\varepsilon_{i,k}|||^2     \nonumber\\
&\leq |||(1-\alpha_{i,k})(\hat{x}_{i,k}-x^*)+\alpha_{i,k}(T_{i,k}-x^*)|||^2        \nonumber\\
&\hspace{0.4cm}+\tau_1\alpha_{i,k}|||\varepsilon_{i,k}|||     \nonumber\\
&= (1-\alpha_{i,k})|||\hat{x}_{i,k}-x^*|||^2+\alpha_{i,k}|||T_{i,k}-x^*|||^2        \nonumber\\
&\hspace{0.4cm}-\alpha_{i,k}(1-\alpha_{i,k})|||T_{i,k}-\hat{x}_{i,k}|||^2+\tau_1\alpha_{i,k}|||\varepsilon_{i,k}|||,      \nonumber
\end{align}
where Lemma \ref{l7} has been utilized to get the last equality. Taking the conditional expectation on the above inequality, one has that
\begin{align}
&\mathbb{E}(|||x_{i,k+1}-x^*|||^2|\chi_k)        \nonumber\\
&\leq (1-\alpha_{i,k})|||\hat{x}_{i,k}-x^*|||^2+\alpha_{i,k}\mathbb{E}(|||T_{i,k}-x^*|||^2|\chi_k)        \nonumber\\
&\hspace{0.4cm}-\alpha_{i,k}(1-\alpha_{i,k})\mathbb{E}(|||T_{i,k}-\hat{x}_{i,k}|||^2|\chi_k)         \nonumber\\
&\hspace{0.4cm}+\tau_1\alpha_{i,k}\mathbb{E}(|||\varepsilon_{i,k}|||~|\chi_k)      \nonumber\\
&\leq \sum_{j=1}^N a_{ij,k}|||x_{j,k}-x^*|||^2+\tau_1\alpha_{i,k}\mathbb{E}(|||\varepsilon_{i,k}|||~|\chi_k)      \nonumber\\
&\hspace{0.4cm}-\alpha_{i,k}(1-\alpha_{i,k})\mathbb{E}(|||T_{i,k}-\hat{x}_{i,k}|||^2|\chi_k),            \nonumber
\end{align}
where the last inequality has made use of (\ref{t4.2}) and the convexity of norm $|||\cdot|||^2$. Then, by multiplying $\pi_{i,k+1}$ on both sides of the above inequality and summing over $i\in[N]$, one can obtain that
\begin{align}
&\mathbb{E}(\sum_{i=1}^N\pi_{i,k+1}|||x_{i,k+1}-x^*|||^2|\chi_k)          \nonumber\\
&\leq \sum_{j=1}^N \pi_{j,k}|||x_{j,k}-x^*|||^2+\tau_1(1-\alpha)\sum_{i=1}^N\mathbb{E}(|||\varepsilon_{i,k}|||~|\chi_k)      \nonumber\\
&\hspace{0.4cm}-\underline{\pi}\alpha(1-\alpha)\sum_{i=1}^N\mathbb{E}(|||T_{i,k}-\hat{x}_{i,k}|||^2|\chi_k),            \label{t4.6}
\end{align}
where Lemma \ref{l1} has been applied. Recalling (\ref{t4.5}) and in light of Lemma \ref{l6}, one has that
\begin{align}
\sum_{k\in\mathbb{N}}\sum_{i=1}^N\mathbb{E}(|||T_{i,k}-\hat{x}_{i,k}|||^2|\chi_k)<\infty          \label{t4.7}
\end{align}
yielding that
\begin{align}
\mathbb{E}(|||T_{i,k}-\hat{x}_{i,k}|||^2|\chi_k)\to 0,~~\text{as}~k\to\infty           \label{t4.8}
\end{align}
which, by the law of total expectation, gives rise to
\begin{align}
\mathbb{E}(|||T_{i,k}-\hat{x}_{i,k}|||^2)\to 0,~~\text{as}~k\to\infty.           \label{t4.9}
\end{align}

Consider the iteration (\ref{04}). It can be written in a compact form
\begin{align}
x_{k+1}=(A_k\otimes Id)x_k+r_k,         \label{t4.10}
\end{align}
where $x_k:=col(x_{1,k},\ldots,x_{N,k})$, $r_k:=col(r_{1,k},\ldots,r_{N,k})$, and $r_{i,k}:=\alpha_{i,k}(T_{i,k}-\hat{x}_{i,k})+\alpha_{i,k}\varepsilon_{i,k}$ for $i\in[N]$. In view of (\ref{t4.5}) and (\ref{t4.9}), it follows that $\mathbb{E}(|||r_{i,k}|||^2)\to 0$ and thus $\mathbb{E}(|||r_{k}|||^2)\to 0$. Then, using the same arguments as that of Lemmas 3 and 4 in \cite{xie2018distributed} and applying Lemma \ref{l3}, one has that for all $i\in[N]$
\begin{align}
\mathbb{E}(|||x_{i,k}-\bar{x}_k|||^2)\to 0,~\text{as}~k\to\infty          \label{t4.11}
\end{align}
where $\bar{x}_k:=\sum_{i=1}^N\pi_{i,k}x_{i,k}$ is viewed as a weighted average of $x_{i,k}$'s. By resorting to Markov's inequality, for arbitrary small $\delta>0$, it can be claimed that
\begin{align}
\mathbb{P}(|||x_{i,k}-\bar{x}_k|||^2>\delta)\leq \frac{\mathbb{E}(|||x_{i,k}-\bar{x}_k|||^2)}{\delta},      \nonumber
\end{align}
which, together with (\ref{t4.11}), implies that
\begin{align}
|||x_{i,k}-\bar{x}_k|||^2\to 0,~\text{a.s.}      \label{t4.12}
\end{align}

Now, combining (\ref{t4.b}) with (\ref{t4.8}) leads to that
\begin{align}
\|F_i(\hat{x}_{i,k})-\hat{x}_{i,k}\|^2\to 0,~\text{as}~k\to\infty        \nonumber
\end{align}
further yielding, by the norm equivalence, that
\begin{align}
|||F_i(\hat{x}_{i,k})-\hat{x}_{i,k}|||^2\to 0,~\text{as}~k\to\infty.              \label{t4.14}
\end{align}

Further, one can have that for all $i\in[N]$
\begin{align}
&|||F_i(\bar{x}_k)-\bar{x}_k|||^2         \nonumber\\
&\leq 3|||F_i(\bar{x}_k)-F_i(\hat{x}_{i,k})|||^2+3|||F_i(\hat{x}_{i,k})-\hat{x}_{i,k}|||^2        \nonumber\\
&\hspace{0.4cm}+3|||\hat{x}_{i,k}-\bar{x}_k|||^2       \nonumber\\
&\leq 6|||\hat{x}_{i,k}-\bar{x}_k|||^2+3|||F_i(\hat{x}_{i,k})-\hat{x}_{i,k}|||^2,          \label{t4.15}
\end{align}
where we have exploited $(a+b+c)^2\leq 3(a^2+b^2+c^2)$ for $a,b,c>0$ and the nonexpansive property of $F_i$ in the first and second inequalities, respectively. Meanwhile, by the convexity of norm, it follows that
\begin{align}
|||\hat{x}_{i,k}-\bar{x}_k|||\leq \sum_{j=1}^N a_{ij,k}|||x_{j,k}-\bar{x}_k|||,        \nonumber
\end{align}
which, together with (\ref{t4.12}), results in
\begin{align}
|||\hat{x}_{i,k}-\bar{x}_k|||\to 0,~\text{as}~k\to\infty.        \label{t4.16}
\end{align}

Combining (\ref{t4.14}), (\ref{t4.15}), and (\ref{t4.16}) gives rise to that for all $i\in[N]$
\begin{align}
|||F_i(\bar{x}_k)-\bar{x}_k|||\to 0,~\text{as}~k\to\infty.          \label{t4.16a}
\end{align}

Finally, following the same reasoning as that between (\ref{t1.15}) and (\ref{t1.16}), the a.s. weak convergence of $x_{i,k}$'s to a common point in $X^*$ in Theorem \ref{t4} can be concluded.

It remains to prove the convergence result in (\ref{05}). This can be similarly done as that of (\ref{12}) using (\ref{t4.7}), (\ref{t4.b}), Lemma \ref{l8}, (\ref{t4.5}), and the law of total expectation. This ends the proof.
\hfill\rule{2mm}{2mm}

{\em Proof of Theorem \ref{t5}:} Let us denote by $D_{S}(x)$ the distance from a vector $x\in\mathcal{H}$ to a set $S$ in space $(\mathcal{H},|||\cdot|||)$. Let $s_{i,k}:=\sum_{j=1}^N a_{ij,k}P_{X^*}(x_{j,k})$ and there exists $\tau_{2}\in(0,\infty)$ such that for all $k\in\mathbb{N},i\in[N]$
\begin{align}
\tau_2&\geq 2|||(1-\alpha_{i,k})(\hat{x}_{i,k}-s_{i,k})      \nonumber\\
&\hspace{1.8cm}+\alpha_{i,k}(T_{i,k}-s_{i,k})|||+\alpha_{i,k}|||\varepsilon_{i,k}|||,~~a.s.     \nonumber
\end{align}
due to the boundedness of $x_{i,k}$'s.

Then, in light of (\ref{04}), it can be derived that
\begin{align}
&D_{X^*}(x_{i,k+1})         \nonumber\\
&\leq |||x_{i,k+1}-s_{i,k}|||^2         \nonumber\\
&=|||(1-\alpha_{i,k})(\hat{x}_{i,k}-s_{i,k})+\alpha_{i,k}(T_{i,k}-s_{i,k})+\alpha_{i,k}\varepsilon_{i,k}|||^2       \nonumber\\
&\leq |||(1-\alpha_{i,k})(\hat{x}_{i,k}-s_{i,k})+\alpha_{i,k}(T_{i,k}-s_{i,k})|||^2            \nonumber\\
&\hspace{0.4cm}+\tau_2\alpha_{i,k}|||\varepsilon_{i,k}|||       \nonumber\\
&=(1-\alpha_{i,k})|||\hat{x}_{i,k}-s_{i,k}|||^2+\alpha_{i,k}|||T_{i,k}-s_{i,k}|||^2           \nonumber\\
&\hspace{0.4cm}-\alpha_{i,k}(1-\alpha_{i,k})|||T_{i,k}-\hat{x}_{i,k}|||^2+\tau_2\alpha_{i,k}|||\varepsilon_{i,k}|||,         \label{t5.1}
\end{align}
where the last equality has invoked Lemma \ref{l7}. Next, as similarly done for (\ref{t4.2}), it can obtain that
\begin{align}
\mathbb{E}(|||T_{i,k}-s_{i,k}|||^2|\chi_k)\leq |||\hat{x}_{i,k}-s_{i,k}|||^2.        \label{t5.2}
\end{align}

Consequently, by multiplying $\pi_{i,k+1}$ on both sides of (\ref{t5.1}), summing over $i\in[N]$, using the convexity of $|||\cdot|||$, and taking the conditional expectation along with (\ref{t4.a}), (\ref{t4.b}), (\ref{t5.2}) and $\|\cdot\|^2\geq p_0|||\cdot|||^2$, one can get that
\begin{align}
&\mathbb{E}(D_{k+1}^2|\chi_k)     \nonumber\\
&\leq D_k^2-\frac{\underline{\pi}p_0^2\alpha(1-\alpha)}{4}\sum_{i=1}^N\mathbb{E}(|||F_i(x_{i,k+1})-x_{i,k+1}|||^2|\chi_k)    \nonumber\\
&\hspace{0.4cm}+\tau_3\sum_{i=1}^N\sqrt{\mathbb{E}(|||\varepsilon_{i,k}|||^2|\chi_k)},          \label{t5.3}
\end{align}
where the parameters $D_k^2:=\sum_{i=1}^N \pi_{i,k} D_{X^*}^2(x_{i,k})$ and the existence of $\tau_3\in(0,\infty)$ is guaranteed by the boundedness of $x_{i,k}$'s, with $\tau_3\geq \tau_2(1-\alpha)+4\underline{\pi}p_0\alpha(1-\alpha)^3\sqrt{\mathbb{E}(|||\epsilon_{i,k}|||^2|\chi_k)}$ a.s. for all $k\in\mathbb{N},i\in[N]$. Using (\ref{t2.9}) in (\ref{t5.3}), one can obtain that
\begin{align}
&\mathbb{E}(D_{k+1}^2|\chi_k)     \nonumber\\
&\leq D_k^2-\frac{\underline{\pi}p_0^2\alpha(1-\alpha)}{8}\sum_{i=1}^N\mathbb{E}(|||F_i(\bar{x}_{k+1})-\bar{x}_{k+1}|||^2|\chi_k)    \nonumber\\
&\hspace{0.4cm}+\sum_{i=1}^N\mathbb{E}(|||x_{i,k+1}-\bar{x}_{k+1}|||^2|\chi_k)        \nonumber\\
&\hspace{0.4cm}+\tau_3\sum_{i=1}^N\sqrt{\mathbb{E}(|||\varepsilon_{i,k}|||^2|\chi_k)}.          \label{t5.4}
\end{align}

In view of (\ref{06}), it can be derived that $\sum_{i=1}^N|||F_i(\bar{x}_{k+1})-\bar{x}_{k+1}|||^2\geq \frac{p_0}{N\nu^2}D_{X^*}^2(\bar{x}_{k+1})$, which, together with (\ref{t2.13}) and (\ref{t5.4}), results in
\begin{align}
\eta_1 \mathbb{E}(D_{k+1}^2|\chi_k)&\leq D_k^2+\tau_4\sum_{i=1}^N\mathbb{E}(|||x_{i,k+1}-\bar{x}_{k+1}|||^2|\chi_k)    \nonumber\\
&\hspace{0.4cm}+\tau_3\sum_{i=1}^N\sqrt{\mathbb{E}(|||\varepsilon_{i,k}|||^2|\chi_k)},        \label{t5.5}
\end{align}
where
\begin{align}
\eta_1&:=1+\frac{\underline{\pi}p_0^2\alpha(1-\alpha)}{16N^2\nu^2},         \nonumber\\
\tau_4&:=\underline{\pi}p_0 \alpha(1-\alpha)(1+\frac{p_0}{8N^2\nu^2}).     \nonumber
\end{align}

It is easy to verify that (\ref{t2.24}) still holds in the expectation sense. Thus, by taking the expectation on both sides of (\ref{t5.5}), one has that
\begin{align}
\eta_1 \mathbb{E}(D_{k+1}^2)&\leq \eta_2\sup_{\lfloor\frac{k+1}{2}\rfloor\leq l\leq k}\mathbb{E}(D_l^2)+\eta_1 e_k',        \label{t5.6}
\end{align}
where $\eta_2:=24\tau_4\alpha_c^2N^3\varpi^2\xi^2/(\underline{\pi}(1-\xi)^2)$ and
\begin{align}
e_k'&:=\frac{1}{\eta_1}\bigg(3\tau_4N^3\varpi^2\xi^{2(k+1)}\mathbb{E}(\|x_0-\bar{x}_0\|^2)         \nonumber\\
&\hspace{1.2cm}+\frac{3\tau_4N^3\varpi^2\xi^{k+3}}{(1-\xi)^2}\sup_{l\in\mathbb{N}}\mathbb{E}(\|r_l\|^2|\chi_k)     \nonumber\\
&\hspace{1.2cm}+\frac{6\tau_4N^3\varpi^2\xi^2}{(1-\xi)^2}\sup_{\lfloor\frac{k+1}{2}\rfloor\leq l\leq k}\mathbb{E}(\|\varepsilon_l\|^2|\chi_k)\bigg).       \nonumber
\end{align}
Therefore, letting $\eta:=\eta_2/\eta_1$ with $\eta\in(0,1)$ under (\ref{07}), it can be concluded that
\begin{align}
\mathbb{E}(D_{k+1}^2)&\leq \eta\sup_{\lfloor\frac{k+1}{2}\rfloor\leq l\leq k}\mathbb{E}(D_l^2)+ e_k',        \label{t5.7}
\end{align}

In the end, invoking the similar argument for (\ref{t2.28}), the conclusions of this theorem can be established. This ends the proof.
\hfill\rule{2mm}{2mm}

\section{Conclusion}\label{s6}

This paper has investigated the problem of seeking a common fixed point for a family of nonexpansive operators over a time-varying multi-agent network in real Hilbert spaces, where each operator is only privately and approximately known by individual agent. In order to deal with the problem, a distributed algorithm, called D-IKM iteration, has been developed, which is shown to be weakly convergent to a common fixed point of the collection of operators, and furthermore, convergent with the rate $O(1/k^{\ln(1/\xi)})$ under the (bounded) linear regularity assumption. To further make this algorithm more implementable in practice, another scenario, where only a random part of coordinate (instead of the entire coordinate) is activated and updated for each agent at each iteration, has been studied. Another distributed algorithm, called D-IBKM, has been accordingly proposed along with the convergence analysis similar to the D-IKM iteration case, but in the sense of almost surely. In addition, a novel concept, i.e., bounded power regularity for a family of operators, has been introduced, which is more relaxed than the counterparts for an operator and a family of sets. It is shown that the convergence rate $O(1/k^{\ln(1/\xi)})$ can still be ensured under the assumption of the new concept. Regarding future work, it is interesting to consider the asynchronous case, i.e., all agents have their own local clocks, and to further study the convergence rate under the (bounded) power regularity.

%





\begin{thebibliography}{10}

\bibitem{bauschke2017convex}
H.~H. Bauschke and P.~L. Combettes, \emph{{Convex Analysis and Monotone
  Operator Theory in Hilbert Spaces}, 2nd ed}.\hskip 1em plus 0.5em minus
  0.4em\relax Springer, New York, 2017.

\bibitem{cegielski2012iterative}
A.~Cegielski, \emph{Iterative Methods for Fixed Point Problems in Hilbert
  Spaces}.\hskip 1em plus 0.5em minus 0.4em\relax Springer, Heidelberg, 2012,
  vol. 2057.

\bibitem{eckstein1992douglas}
J.~Eckstein and D.~P. Bertsekas, ``{On the Douglas-Rachford splitting method
  and the proximal point algorithm for maximal monotone operators},''
  \emph{Mathematical Programming}, vol.~55, no. 1-3, pp. 293--318, 1992.

\bibitem{attouch2010parallel}
H.~Attouch, L.~M. Briceno-Arias, and P.~L. Combettes, ``A parallel splitting
  method for coupled monotone inclusions,'' \emph{SIAM Journal on Control and
  Optimization}, vol.~48, no.~5, pp. 3246--3270, 2010.

\bibitem{iiduka2016convergence}
H.~Iiduka, ``Convergence analysis of iterative methods for nonsmooth convex
  optimization over fixed point sets of quasi-nonexpansive mappings,''
  \emph{Mathematical Programming}, vol. 159, no. 1-2, pp. 509--538, 2016.

\bibitem{borwein2017convergence}
J.~M. Borwein, G.~Li, and M.~K. Tam, ``Convergence rate analysis for averaged
  fixed point iterations in common fixed point problems,'' \emph{SIAM Journal
  on Optimization}, vol.~27, no.~1, pp. 1--33, 2017.

\bibitem{haskell2017random}
W.~B. Haskell and R.~Jain, ``A random monotone operator framework for strongly
  convex stochastic optimization,'' in \emph{Proceedings of IEEE 56th
  Conference on Decision and Control}, Melbourne, Australia, 2017, pp.
  3763--3768.

\bibitem{yi2017distributed}
P.~Yi and L.~Pavel, ``{A distributed primal-dual algorithm for computation of
  generalized Nash equilibria via operator splitting methods},'' in
  \emph{Proceedings of 56th Conference on Decision and Control}, Melbourne,
  Australia, 2017, pp. 3841--3846.

\bibitem{banjac2018tight}
G.~Banjac and P.~J. Goulart, ``Tight global linear convergence rate bounds for
  operator splitting methods,'' \emph{IEEE Transactions on Automatic Control},
  vol.~63, no.~12, pp. 4126--4139, 2018.

\bibitem{xu2018bregman}
J.~Xu, S.~Zhu, Y.~C. Soh, and L.~Xie, ``{A Bregman splitting scheme for
  distributed optimization over networks},'' \emph{IEEE Transactions on
  Automatic Control}, vol.~63, no.~11, pp. 3809--3824, 2018.

\bibitem{ren2010distributed}
W.~Ren and Y.~Cao, \emph{{Distributed Coordination of Multi-Agent Networks:
  Emergent Problems, Models, and Issues}}.\hskip 1em plus 0.5em minus
  0.4em\relax London, U.K.: Springer-Verlag, 2010.

\bibitem{fullmer2018distributed}
D.~Fullmer and A.~S. Morse, ``A distributed algorithm for computing a common
  fixed point of a finite family of paracontractions,'' \emph{IEEE Transactions
  on Automatic Control}, vol.~63, no.~9, pp. 2833--2843, 2018.

\bibitem{fullmer2016asynchronous}
D.~Fullmer, J.~Liu, and A.~S. Morse, ``An asynchronous distributed algorithm
  for computing a common fixed point of a family of paracontractions,'' in
  \emph{Proceedings of 55th Conference on Decision and Control}, Las Vegas,
  USA, 2016, pp. 2620--2625.

\bibitem{mou2015distributed}
S.~Mou, J.~Liu, and A.~S. Morse, ``A distributed algorithm for solving a linear
  algebraic equation,'' \emph{IEEE Transactions on Automatic Control}, vol.~60,
  no.~11, pp. 2863--2878, 2015.

\bibitem{wang2016distributed}
L.~Wang, D.~Fullmer, and A.~S. Morse, ``A distributed algorithm with an
  arbitrary initialization for solving a linear algebraic equation,'' in
  \emph{Proceedings of American Control Conference}, Boston, MA, USA, 2016, pp.
  1078--1081.

\bibitem{wang2017further}
X.~Wang, S.~Mou, and D.~Sun, ``Further discussions on a distributed algorithm
  for solving linear algebra equations,'' in \emph{Proceedings of American
  Control Conference}, Seattle, USA, 2017, pp. 4274--4278.

\bibitem{wang2018distributed}
P.~Wang, W.~Ren, and Z.~Duan, ``Distributed algorithm to solve a system of
  linear equations with unique or multiple solutions from arbitrary
  initializations,'' \emph{IEEE Transactions on Control of Network Systems}, in
  press, doi: 10.1109/TCNS.2018.2797805, 2018.

\bibitem{alaviani2018distributed}
S.~S. Alaviani and N.~Elia, ``A distributed algorithm for solving linear
  algebraic equations over random networks,'' \emph{arXiv preprint
  arXiv:1809.07955}, 2018.

\bibitem{liu2017distributed}
J.~Liu, D.~Fullmer, A.~Nedi{\'c}, T.~Ba{\c{s}}ar, and A.~S. Morse, ``A
  distributed algorithm for computing a common fixed point of a family of
  strongly quasi-nonexpansive maps,'' in \emph{Proceedings of American Control
  Conference}, Seattle, USA, 2017, pp. 686--690.

\bibitem{cominetti2014rate}
R.~Cominetti, J.~A. Soto, and J.~Vaisman, ``{On the rate of convergence of
  Krasnosel'ski\u{\i}-Mann iterations and their connection with sums of
  Bernoullis},'' \emph{Israel Journal of Mathematics}, vol. 199, no.~2, pp.
  757--772, 2014.

\bibitem{liang2016convergence}
J.~Liang, J.~Fadili, and G.~Peyr{\'e}, ``Convergence rates with inexact
  non-expansive operators,'' \emph{Mathematical Programming}, vol. 159, no.
  1-2, pp. 403--434, 2016.

\bibitem{matsushita2017convergence}
S.~Matsushita, ``{On the convergence rate of the Krasnosel'ski\u{\i}-Mann
  iteration},'' \emph{Bulletin of the Australian Mathematical Society},
  vol.~96, no.~1, pp. 162--170, 2017.

\bibitem{kanzow2017generalized}
C.~Kanzow and Y.~Shehu, ``{Generalized Krasnosel'ski\u{\i}-Mann-type iterations
  for nonexpansive mappings in Hilbert spaces},'' \emph{Computational
  Optimization and Applications}, vol.~67, no.~3, pp. 595--620, 2017.

\bibitem{bravo2018sharp}
M.~Bravo and R.~Cominetti, ``Sharp convergence rates for averaged nonexpansive
  maps,'' \emph{Israel Journal of Mathematics}, vol. 227, no.~1, pp. 163--188,
  2018.

\bibitem{bravo2018rates}
M.~Bravo, R.~Cominetti, and M.~Pavez-Sign{\'e}, ``{Rates of convergence for
  inexact Krasnosel'ski\u{\i}-Mann iterations in Banach spaces},''
  \emph{Mathematical Programming}, in press,
  https://doi.org/10.1007/s10107-018-1240-1, 2018.

\bibitem{shehu2018convergence}
Y.~Shehu, ``{Convergence rate analysis of inertial Krasnosel'ski\u{\i}-Mann
  type iteration with applications},'' \emph{Numerical Functional Analysis and
  Optimization}, vol.~39, no.~10, pp. 1077--1091, 2018.

\bibitem{mann1953mean}
W.~R. Mann, ``Mean value methods in iteration,'' \emph{Proceedings of the
  American Mathematical Society}, vol.~4, no.~3, pp. 506--510, 1953.

\bibitem{krasnosel1955two}
M.~A. Krasnosel'ski\u{\i}, ``Two comments on the method of successive
  approximations,'' \emph{Uspekhi Matematicheskikh Nauk}, vol.~10, pp.
  123--127, 1955.

\bibitem{rockafellar1976monotone}
R.~T. Rockafellar, ``Monotone operators and the proximal point algorithm,''
  \emph{SIAM Journal on Control and Optimization}, vol.~14, no.~5, pp.
  877--898, 1976.

\bibitem{passty1979ergodic}
G.~B. Passty, ``{Ergodic convergence to a zero of the sum of monotone operators
  in Hilbert space},'' \emph{Journal of Mathematical Analysis and
  Applications}, vol.~72, pp. 383--290, 1979.

\bibitem{paeceman1955numerical}
D.~Paeceman and H.~Rachford, ``The numerical solution of parabolic and elliptic
  equations,'' \emph{Journal of the Society for Industrial and Applied
  Mathematics}, vol.~3, no.~1, pp. 28--41, 1955.

\bibitem{douglas1956numerical}
J.~Douglas and H.~H. Rachford, ``On the numerical solution of heat conduction
  problems in two and three space variables,'' \emph{Transactions of the
  American Mathematical Society}, vol.~82, no.~2, pp. 421--439, 1956.

\bibitem{lions1979splitting}
P.~L. Lions and B.~Mercier, ``Splitting algorithms for the sum of two nonlinear
  operators,'' \emph{SIAM Journal on Numerical Analysis}, vol.~16, no.~6, pp.
  964--979, 1979.

\bibitem{gabay1975dual}
D.~Gabay and B.~Mercier, ``A dual algorithm for the solution of nonlinear
  variational problems via finite element approximation,'' \emph{Computers \&
  Mathematics with Applications}, vol.~2, pp. 17--40, 1976.

\bibitem{davis2017three}
D.~Davis and W.~Yin, ``A three-operator splitting scheme and its optimization
  applications,'' \emph{Set-Valued and Variational Analysis}, vol.~25, no.~4,
  pp. 829--858, 2017.

\bibitem{reich1979weak}
S.~Reich, ``{Weak convergence theorems for nonexpansive mappings in Banach
  spaces},'' \emph{Journal of Mathematical Analysis and Applications}, vol.~67,
  no.~2, pp. 274--276, 1979.

\bibitem{cegielski2015application}
A.~Cegielski, ``Application of quasi-nonexpansive operators to an iterative
  method for variational inequality,'' \emph{SIAM Journal on Optimization},
  vol.~25, no.~4, pp. 2165--2181, 2015.

\bibitem{xie2018distributed}
P.~Xie, K.~You, R.~Tempo, S.~Song, and C.~Wu, ``Distributed convex optimization
  with inequality constraints over time-varying unbalanced digraphs,''
  \emph{IEEE Transactions on Automatic Control}, vol.~63, no.~12, pp.
  4331--4337, 2018.

\bibitem{bauschke2015linear}
H.~H. Bauschke, D.~Noll, and H.~M. Phan, ``Linear and strong convergence of
  algorithms involving averaged nonexpansive operators,'' \emph{Journal of
  Mathematical Analysis and Applications}, vol. 421, no.~1, pp. 1--20, 2015.

\bibitem{dontchev2009implicit}
A.~L. Dontchev and R.~T. Rockafellar, \emph{Implicit Functions and Solution
  Mappings: A View from Variational Analysis}.\hskip 1em plus 0.5em minus
  0.4em\relax Springer, New York, 2009.

\bibitem{ortega1970iterative}
J.~M. Ortega and W.~C. Rheinboldt, \emph{Iterative Solution of Nonlinear
  Equations in Several Variables}.\hskip 1em plus 0.5em minus 0.4em\relax
  Academic Press, New York, 1970.

\bibitem{combettes2004solving}
P.~L. Combettes, ``Solving monotone inclusions via compositions of nonexpansive
  averaged operators,'' \emph{Optimization}, vol.~53, no. 5-6, pp. 475--504,
  2004.

\bibitem{combettes2015stochastic}
P.~L. Combettes and J.-C. Pesquet, ``{Stochastic quasi-Fej{\'e}r
  block-coordinate fixed point iterations with random sweeping},'' \emph{SIAM
  Journal on Optimization}, vol.~25, no.~2, pp. 1221--1248, 2015.

\bibitem{nedic2015distributed}
A.~Nedi{\'c} and A.~Olshevsky, ``Distributed optimization over time-varying
  directed graphs,'' \emph{IEEE Transactions on Automatic Control}, vol.~60,
  no.~3, pp. 601--615, 2015.

\bibitem{robbins1971martin}
H.~Robbins and D.~Siegmund, ``A convergence theorem for nonnegative almost
  supermartingales and some applications,'' in \emph{Proceedings of Optimizing
  Methods in Statistics}, Ohio, USA, 1971, pp. 233--257.

\end{thebibliography}


\end{document}